\documentclass[11pt]{article}

\usepackage{amsmath, amssymb, amsthm, verbatim,enumerate,bbm,color} 
\numberwithin{equation}{section}
\usepackage{float}

\usepackage{mathrsfs}

\usepackage{tikz}
\usetikzlibrary{calc}
\usetikzlibrary{math}
\usetikzlibrary{decorations,decorations.markings,decorations.pathreplacing,decorations.pathmorphing}
\usetikzlibrary{fadings}
\usetikzlibrary{fit}
\usetikzlibrary{matrix}
\usetikzlibrary{patterns}
\usetikzlibrary{positioning}
\usetikzlibrary{shadows}
\usetikzlibrary{shapes,shapes.geometric}
\usetikzlibrary{trees}
\usetikzlibrary{overlay-beamer-styles}
\usepackage{tikzsymbols}

\usepackage{dsfont}

\usepackage[colorlinks,citecolor=blue,bookmarks=true]{hyperref}
\usepackage{cleveref}

\date{\today}
\allowdisplaybreaks

\theoremstyle{plain}
\newtheorem{theorem}{Theorem}[section]
\newtheorem{lemma}[theorem]{Lemma}
\newtheorem{claim}[theorem]{Claim}

\newtheorem{problem}[theorem]{Problem}

\newtheorem{definition}[theorem]{Definition}

\newcommand{\Bin}{\ensuremath{\textrm{Bin}}}

\newcommand{\disc}{\ensuremath{\textrm{disc}}}

\makeatletter
\def\moverlay{\mathpalette\mov@rlay}
\def\mov@rlay#1#2{\leavevmode\vtop{%
		\baselineskip\z@skip \lineskiplimit-\maxdimen
		\ialign{\hfil$\m@th#1##$\hfil\cr#2\crcr}}}
\newcommand{\charfusion}[3][\mathord]{
	#1{\ifx#1\mathop\vphantom{#2}\fi
		\mathpalette\mov@rlay{#2\cr#3}
	}
	\ifx#1\mathop\expandafter\displaylimits\fi}
\makeatother

\makeatletter
\renewenvironment{proof}[1][\proofname]
{\par\pushQED{\qed}
	\normalfont\topsep6\p@\@plus6\p@\relax\trivlist
	\item[\hskip\labelsep\bfseries#1\@addpunct{.}]
	\ignorespaces}
{\popQED\endtrivlist\@endpefalse}

\makeatother

\newcommand{\eps}{\varepsilon}

\RequirePackage{color}\definecolor{RED}{rgb}{1,0,0}\definecolor{BLUE}{rgb}{0,0,1} 

\title{Subgraph discrepancies in the complete graph}

\author{
Micha Christoph\thanks{Department of Mathematics, ETH Z\"urich, Switzerland. \emph{Email}: \href{micha.christoph@math.ethz.ch.}{\tt micha.christoph@math.ethz.ch}. Research supported by the SNSF Ambizione Grant No. 216071.}
\and 
Lior Gishboliner\thanks{Department of Mathematics, University of Toronto, Canada. \emph{Email}: \href{mailto:lior.gishboliner@utoronto.ca}{\tt lior.gishboliner@utoronto.ca}.
Research supported in part by the NSERC Discovery Grant ``Problems in Extremal and Probabilistic Combinatorics".
}
\and 
Michael Krivelevich\thanks{Department of Mathematics, Tel Aviv University, Israel. \emph{Email}: \href{krivelev@tauex.tau.ac.il }{\tt krivelev@tauex.tau.ac.il}. Research supported in part by NSF-BSF grant 2023688.}
}

\date{}

\begin{document}

\maketitle

\begin{abstract}
    Given a 2-edge-coloring 
    $f : E(K_n) \rightarrow \{\pm 1\}$, the {\em discrepancy} of a subgraph $F \subseteq K_n$ is defined as $\left| \sum_{e \in E(F)} f(e) \right|$. Erd\H{o}s, F\"uredi, Loebl and S\'os showed that if $F$ is an $n$-vertex tree with maximum degree at most 
    $(1-\varepsilon)n$, then every 2-coloring of $K_n$ has a copy of $F$ with discrepancy $\Omega(\varepsilon)n$. We extend this result by showing that the same conclusion holds for every $n$-vertex graph with maximum degree at most $(1-\varepsilon)n$ and no isolated vertices. We also show that for every $d$-regular $n$-vertex graph $F$ with $d \leq (1-\varepsilon)n$, every 2-coloring of $K_n$ has a copy of $F$ with discrepancy $\Omega(\sqrt{\varepsilon d}) \cdot n$. The dependence on $d$ and $n$ is best possible.

    Finally, we consider specific graphs $F$, namely $K_r$-factors and 2-factors. For each such graph $F$, we determine the optimal constant $\lambda$ such that every 2-coloring of $K_n$ has a copy of $F$ with discrepancy at least $(\lambda + o(1))n$. 
\end{abstract}

\section{Introduction}\label{sec:introduction}

Combinatorial discrepancy is a classical topic in discrete mathematics, dealing with balanced colorings of hypergraphs. More precisely, for a hypergraph $H$ and a 2-coloring $f : V(H) \rightarrow \{\pm 1\}$, we define $D(f) = \max_{e \in E(H)} |\sum_{x \in e} f(x)|$. The {\em discrepancy} of $H$ is then defined as $\min_{f} D(f)$. In other words, the discrepancy is the maximum imbalance (on some edge) which is guaranteed to exist in every 2-coloring of $V(H)$. 

A widely studied variant of the above setting is when the hypergraph $H$ corresponds to a set of subgraphs of a graph. Namely, $V(H)$ is the set of {\em edges} of a graph $G$, and $E(H)$ is the set of subgraphs of $G$ of a given type, e.g., all subgraphs of $G$ isomorphic to a given graph $F$. Two early notable works on problems of this type are that of Erd\H{o}s and Spencer \cite{ES:72}, who studied the case of cliques, and that of Erd\H{o}s, F\"uredi, Loebl and S\'os \cite{EFLS:95}, who studied spanning trees. We shall return to the main result of \cite{EFLS:95} shortly. 

In recent years, graph discrepancy problems have witnessed renewed interest, following the papers of Balogh, Csaba, Jing and Pluh\'ar \cite{BCJP:20} and
Balogh, Csaba, Pluh\'ar and Treglown \cite{BCPT:21}. By now, problems of this type have been studied for a variety of structures, such as perfect matchings and Hamilton cycles~\cite{BCJP:20, BCL, FHLT:21,GKM_Hamilton, GKM_trees}, spanning trees~\cite{GKM_trees,HLMP}, $H$-factors~\cite{BCPT:21,BCG:23}, powers of Hamilton cycles~\cite{Bradac:22}, 1-factorizations \cite{AHIL}, and general bounded-degree graphs \cite{BPPR,HLMP}. Recently, analogous problems for uniform hypergraphs have also been studied \cite{BTZ-G:24,GGS_Hamilton,GGS_STS,HLMPSTZ24+,LMX24+}.

In this paper we study the discrepancy problem for general spanning guest graphs $F$. 
Repeating the definition of discrepancy for our setting, 
for a coloring $f : E(K_n) \rightarrow \{\pm 1\}$, the discrepancy of a subgraph $F$ of $K_n$ is 
$|\sum_{e \in E(F)}f(e)|$. In other words, $F$ has discrepancy at least $t$ if one of the two colors appears on at least $\frac{e(F) + t}{2}$ edges. 
Our starting point is the result of Erd\H{o}s, F\"uredi, Loebl and S\'os \cite{EFLS:95}, who proved that for every $n$-vertex tree $T$ with maximum degree at most $(1-\varepsilon)n$, every 2-coloring of $K_n$ has a copy of $F$ with discrepancy 
$\Omega(\varepsilon n)$. This was recently extended to the multicolor setting in \cite{HLMP}. Here we generalize this result to hold for any graph $F$ with maximum degree at most $(1-\varepsilon)n$ and no isolated vertices.
\begin{theorem}\label{thm:max degree}
There exists an absolute constant $c > 0$ such that the following holds. For every $\varepsilon > 0$, every large enough $n$ and every $n$-vertex graph $F$ with $\Delta(F) \leq (1-\varepsilon)n$ and no isolated vertices, every 2-coloring of $K_n$ has a copy of $F$ with discrepancy at least $c \cdot \varepsilon n$. 
\end{theorem}

Theorem \ref{thm:max degree} is tight, even if the graph $F$ is required to have large {\em average} degree. Indeed, take $F$ to consist of a star with $(1-\varepsilon)n$ leaves, a clique of size $k$, and a path on the remaining $\varepsilon n - 1 - k$ vertices. Now color $K_n$ randomly with 2 colors. Consider any copy of $F$ in this coloring. Since every vertex in this coloring has degree $(\frac{1}{2} + o(1))n$ in each color, the star in $F$ has at least $(\frac{1-\varepsilon}{2} - o(1))n$ edges in each color. Also, using Hoeffding's inequality, it is easy to show that in a random 2-coloring, w.h.p.~every clique of size $k$ has discrepancy $O(k^{3/2}\log n)$. Thus, if 
$k \ll (n/\log n)^{2/3}$, then the $k$-clique in $F$ has discrepancy $o(n)$. It follows that every copy of $F$ has discrepancy $O(\varepsilon n)$. Reiterating, this works even if $F$ has average degree $\Theta_{\varepsilon}(n^{1/3}/\log^{4/3} n)$. We wonder whether the answer changes if the average degree of $F$ is larger than that.  

Next, we consider regular graphs $F$, and show that $d$-regular graphs always have discrepancy $\Omega(\sqrt{d}n)$.
\begin{theorem}\label{thm:d-regular}
There is an absolute constant $c > 0$ such that the following holds. Let $\varepsilon>0$, let $n$ be sufficiently large, and let $F$ be a $d$-regular $n$-vertex graph with $d \leq (1-\varepsilon)n$. Then in every 2-coloring of $K_n$, there is a copy of $F$ with discrepancy at least $c\sqrt{\varepsilon d}n$.
\end{theorem}

The bound in Theorem \ref{thm:d-regular} is best possible. Indeed, it is well-known that for any $n,d$ with $nd$ even, there exists a $d$-regular $n$-vertex graph $F$ such that for every partition $V(F) = U \cup V$ with $|U| = \lfloor \frac{n}{2} \rfloor$, it holds that 
$\big| e(U,V) - \frac{e(F)}{2} \big| \leq O(\sqrt{d}n)$. For example, this holds w.h.p.~for a random $d$-regular graph, or more generally for any $(n,d,\lambda)$-graph with $\lambda = O(\sqrt{d})$ (by the expander mixing lemma); see \cite[Section 2.4]{KS_Survey}. Now color $K_n$ with 2 colors such that one of the colors forms a balanced complete bipartite graph. Then every copy of $F$ in this coloring has discrepancy $O(\sqrt{d}n)$.

Next, we consider special types of guest graphs of $F$: $K_k$-factors and 2-factors. 
In the following results, we switch away from the language of discrepancy and instead consider the size of the most popular color among the edges of $F$. This readily translates to a discrepancy result: If a copy of $F$ has at least $m$ edges in the most popular color, then this copy has discrepancy at least $2m-e(F)$. 

The minimum degree threshold for having $K_k$-factors with linear discrepancy was determined in \cite{BCPT:21}. Here we ask a different question: We consider the complete graph instead of graphs of large minimum degree, and we would like to determine (asymptotically) precisely the (minimum possible) discrepancy of $K_k$-factors in a coloring of $K_n$. 
The case $k=2$ (i.e., perfect matchings) was resolved in \cite{GKM_Hamilton}. Here we handle the general case. 
To state our result, we need the following definition. 
\begin{definition}[bipartite construction]\label{def:bipartite construction}
The {\em bipartite construction with ratio $\rho$} is the following coloring of $K_n$. Partition $V(K_n)$ into sets $X,Y$ with $|X| =\rho n$, and color all edges touching $X$ with one color and all edges inside $Y$ with the other color. 
\end{definition}
It turns out that the bipartite construction is extremal for $K_k$-factors. Formally,  
for $k \geq 2$, let $\lambda_k$ be the supremum of all $\lambda$ such that for every $n$ divisible by $k$ and for every $\rho \in [0,1]$, the $n$-vertex bipartite construction with ratio $\rho$ has a $K_k$-factor with at least $\lambda n$ edges of the same color. See Table \ref{fig:lambda_k} in Section \ref{subsec:K_k-factors} for some values of $\lambda_k$.
\begin{theorem}\label{thm: Kk-factors}
    For every $k \geq 2$ and every $n$ divisible by $k$, every 2-coloring of $K_n$ has a $K_k$-factor with at least $(\lambda_k - o(1))n$ edges of the same color.
\end{theorem}

We now move on to 2-factors (i.e., 2-regular graphs). The case of Hamilton cycles was handled in \cite{GKM_Hamilton}, where it was shown that every 2-coloring of $K_n$ has a Hamilton cycle with at least $(2/3-o(1))n$ edges of the same color. Our next result shows that this bound holds for any $2$-factor $F$. 
\begin{theorem}\label{prop:2-factor}
    For every $n$ and every $n$-vertex 2-factor $F$, every 2-coloring of $K_n$ has a copy of $F$ with at least $(2/3-o(1))n$ edges of the same color.  
\end{theorem}

The constant $\frac{2}{3}$ in Theorem \ref{prop:2-factor} is best possible, and the extremal example is again the bipartite construction. Indeed, consider the bipartite construction with ratio $\frac{1}{3}$. Namely, this coloring of $K_n$ consists of two disjoint sets $X,Y$, $|X| = n/3$, such that all edges touching $X$ are red and all edges inside $Y$ are blue. Let $F$ be any $n$-vertex 2-factor in this red/blue-coloring of $K_n$. The number of edges of $F$ contained in $Y$ is at most $|Y| = \frac{2}{3}n$ (since in every subgraph of $F$, the number of edges does not exceed the number of vertices). Also, the number of edges of $F$ touching $X$ is at most $2|X| = \frac{2n}{3}$, since $F$ is 2-regular. Thus, $F$ has at most $\frac{2n}{3}$ edges of each color.  

\section{Proof of Theorems \ref{thm:max degree} and \ref{thm:d-regular}}

We first give an outline of the proofs of Theorems \ref{thm:max degree} and \ref{thm:d-regular}. For the sake of simplicity, we focus on Theorem \ref{thm:d-regular}.
First, we show that if the given red/blue coloring of $K_n$ has a biased bisection $(X,Y)$, namely a bisection where $e(X,Y)$ deviates significantly from $\frac{1}{2}|X||Y|$ in one (and hence both) of the colors, then we can find a copy of $F$ with high discrepancy. This is done by taking a bisection of $F$ whose number of edges deviates significantly from $\frac{e(F)}{2}$ (it is known that it is possible to obtain deviation $\Omega(\sqrt{d}n)$) and embedding this biased bisection of $F$ randomly onto the biased bisection of $G$; see Lemma \ref{lem:cut discrepancy}.

The above allows us to assume that in the red/blue coloring of $K_n$, all bisections are almost unbiased. We then use this assumption to find pairs of vertices of $K_n$ with useful properties. Roughly speaking, we find a partition $V(K_n) = X \cup Y$ and linearly many pairs of vertices $x_i,x'_i \in X$ such that the degrees of $x_i,x'_i$ to $Y$ are close to each other, but on the other hand $|N_Y(x_i) \triangle N_Y(x'_i)|$ is large (see Lemma \ref{lem:host-good}). The usefulness of such pairs is as follows: We fix a partition $V(F) = U \cup V$ to match the partition $X \cup Y$ of $K_n$, and find linearly many pairs of vertices $u_i,u'_i \in U$ again having useful properties. Roughly speaking, $u_i,u'_i$ have similar degrees to $V$ but $|N_V(u_i) \triangle N_V(u'_i)|$ is large. We then map $U$ into $X$, mapping $u_i,u'_i$ onto $x_i,x'_i$ for every $i$, and map $V$ randomly onto $Y$. 
This mapping gives one copy of $F$ in $K_n$. We then obtain a second copy by switching the roles of $u_i$ and $u'_i$ for some of the pairs $i$. 
The point is that for a given index $i$, there is constant probability for the switching to produce an excess $\Omega(\sqrt{d})$ of red edges. This is due to the hypergeometric distribution involved, which has standard deviation $\Theta(\sqrt{d})$. Thus, switching only those pairs which lead to an excess of red edges, we see that this switching produces an $\Omega(n\sqrt{d})$ excess of red edges between $U$ and $V$, and this is the main idea. We still have to consider the effect of switching on the edges inside $U$. If the effect is small (meaning that the number of red edges does not change much), then the above argument suffices. And the complementary case needs to be handled differently. The argument we just outlined is the proof of Lemma \ref{lem:main}.  

The rest of this section is organized as follows.
In Section \ref{subsec:bisections} we consider biased bisections and use them to obtain copies of $F$ with high discrepancy. We also state some known results on the existence of biased bisections in various graphs $F$. In Section \ref{subsec:main lemma} we prove one of our main lemmas, Lemma \ref{lem:main}, whose proof was outlined above. Section \ref{subsec:host} is then devoted to guaranteeing that the assumptions of Lemma \ref{lem:main} hold for the red/blue coloring of $K_n$. Finally, we combine everything in Section \ref{subsec:proof of theorems} to prove Theorems \ref{thm:max degree} and \ref{thm:d-regular}.

Throughout this section, we consider a $\{\pm 1\}$-coloring of $K_n$, and let $G$ denote the graph consisting of the edges of color 1. Then, given an $n$-vertex graph $F$, our goal is to find a copy of $F$ in $K_n$ such that $|E(F) \cap E(G)|$ deviates significantly from $\frac{e(F)}{2}$. We assume that $n$ is large enough wherever \nolinebreak needed.   

\subsection{Discrepancy through cuts}\label{subsec:bisections}

Some of the key ingredients in the proofs of Theorems 
\ref{thm:max degree} and \ref{thm:d-regular} are theorems stating that the graph $F$ has a bisection (i.e., a partition $V(F) = U \cup V$ with $|U| = \lfloor \frac{n}{2} \rfloor$) whose size $e(U,V)$ deviates significantly from the average $e(F)/2$. We need one such statement for connected graphs of maximum degree at most $(1-\varepsilon)n$ and one for $d$-regular graphs (using the former for Theorem \ref{thm:max degree} and the latter for Theorem \ref{thm:d-regular}). 
Here we rely on known results; see Theorems \ref{thm:bisection max degree} and \ref{thm:bisection d-regular} below for these statements and references. 

To exploit the existence of a high-discrepancy bisection in $F$, we need to know that the red/blue coloring of $K_n$ also has a bisection whose size deviates from its expected value. Namely, we need a bisection whose edges are (relatively) far from being colored evenly between red and blue.
The following lemma shows that having such high-discrepancy bisections -- one in $F$ and one in $G$ -- suffices to get a copy of $F$ with high discrepancy. 
\begin{lemma}\label{lem:cut discrepancy}
    The following holds for every large enough $n$.
    Let $t \geq 1$, let $F$ be an $n$-vertex graph, and suppose that there is a partition $V(F) = U \cup V$ with 
    $|U| = \lfloor \frac{n}{2} \rfloor$ and $|e_F(U,V) - \frac{e(F)}{2}| \geq t$. 
    Let $\gamma \in (0,1)$, let $G$ be a subgraph of $K_n$, and suppose that there is a partition $V(G) = X \cup Y$ with 
    $|X| = \lfloor \frac{n}{2} \rfloor$ and $|e_G(X,Y) - \frac{1}{2}|X||Y|| \geq \gamma n^2$. 
    Suppose that $\gamma t \geq \frac{10e(F)}{n}$. 
    Then there is a copy $F_0$ of $F$ in $K_n$ with 
    $\big| |E(F_0) \cap E(G)| - \frac{e(F)}{2} \big| \geq 0.5 \gamma t$. 
\end{lemma}
\noindent
For the proof of Lemma \ref{lem:cut discrepancy}, as well as later on the in paper, we need the following simple lemma.
\begin{lemma}\label{lem: random embedding}
        Let $F$ be an $n$-vertex graph, and let $G$ be a subgraph of $K_n$ with $p\binom{n}{2}$ edges, $p\in[0,1]$. Then there is a copy $F_0$ of $F$ in $K_n$ such that
        $|E(F_0) \cap E(G)| \geq p \cdot e(F)$.
    \end{lemma}
    \begin{proof}
        Let $\phi:V(F)\to V(K_n)$ be a random bijection. The expectation of $|\varphi(E(F)) \cap E(G)|$ is $p \cdot e(F)$.
    \end{proof}

\begin{proof}[Proof of Lemma \ref{lem:cut discrepancy}]
    We would like to assume that $n$ is even. If not, then let $v \in V$ be a vertex with $d_F(v) \leq \frac{2e(F)}{n}$ (such a vertex exists by averaging), and consider $F' = F - v$. Now $|V(F')|$ is even, and we have 
    $\big|e_{F'}(U,V \setminus \{v\}) - \frac{e(F')}{2}\big| \geq 0.9t$. Similarly, deleting a vertex $y \in Y$, we get a graph $G' = G - y$ with $\left| e_{G'}(X,Y \setminus \{y\}) - \frac{1}{2}|X|(|Y|-1) \right| \geq 0.9\gamma n^2$. Now, apply the statement for even $n$ to obtain an embedding $F'_0$ of $F'$ into $K_{n-1}$ with 
    $\big| |E(F'_0) \cap E(G')| - \frac{e(F')}{2} \big| \geq 0.8 \gamma t$ (this is what the proof below will give). Bringing back the vertex $v$ (and mapping it to $y$), we still have 
    $\big| |E(F_0) \cap E(G)| - \frac{e(F)}{2} \big| \geq 0.5\gamma t$. 

    So, assume from now on that $n$ is even, and hence $|U| = |V| = |X| = |Y| = \frac{n}{2}$.
    Let $p = e(G)/\binom{n}{2}$. We may assume that $|p - \frac{1}{2}| \leq \frac{\gamma t}{e(F)}$, because otherwise we are done by Lemma \ref{lem: random embedding} (applied to both $G$ and $\overline{G}$). 
 
    For convenience, put 
    $$
    d_X := \frac{e_G(X)}{\binom{|X|}{2}}, \; \; 
	d_Y := \frac{e_G(Y)}{\binom{|Y|}{2}}, \; \; d_{X,Y} := \frac{e_G(X,Y)}{|X||Y|}.$$ 
    Note that $|d_{X,Y} - \frac{1}{2}| \geq 4\gamma$ because 
	$|e_G(X,Y) - \frac{1}{2}|X||Y|| \geq \gamma n^2$.
	Also, we have
	$$
	\frac{\binom{|X|}{2}}{\binom{n}{2}} \cdot d_X + 
	\frac{\binom{|Y|}{2}}{\binom{n}{2}} \cdot d_Y + 
	\frac{|X||Y|}{\binom{n}{2}} \cdot d_{X,Y} = \frac{e(G)}{\binom{n}{2}} = p. 
	$$
	Using that $|X| = |Y| = \frac{n}{2}$, we get that\footnote{Here, the notation $x = y \pm z$ means that $y-z \leq x \leq y+z$.}
	\begin{equation}\label{eq:density average}
    d_X + d_Y + 2d_{X,Y} = 4p \pm \frac{3}{n}.
	\end{equation}
    The following simple claim is a special case of the Chebyshev sum inequality.
    \begin{claim}\label{claim:Chebyshev sum inequality}
		Let $x,y,u,v$ be real numbers with $x \geq y$ and 
        $u \geq v$. Then 
        $
        \frac{1}{2}(xu+yv) \geq \frac{x+y}{2} \cdot \frac{u+v}{2}.
        $
	\end{claim}
	\begin{proof}
		$
		\frac{1}{2}(xu+yv) - \frac{x+y}{2}\cdot \frac{u+v}{2} = 
		\frac{1}{4}(x-y)(u-v) \geq 0.
		$
	\end{proof}
    Returning to the proof of Lemma \ref{lem:cut discrepancy}, we may assume, without loss of generality, that $d(X) \geq d(Y)$ (else switch the roles of $X$ and $Y$) and similarly 
	$e(U) \geq e(V)$ (else switch the roles of $U$ and $V$).
    Let $f : V(F) \rightarrow V(K_n)$ be chosen uniformly at random among all bijections satisfying $f(U) = X$ and $f(V) = Y$. 
	Putting $F_0 := f(F)$, we have
	\begin{equation}\label{eq:expected number of c-color edges in random embedding}
	\mathbb{E}[|E(F_0) \cap E(G)|] = 
	e(U) \cdot d_X + 
	e(V) \cdot d_Y + 
	e(U,V) \cdot d_{X,Y}.
	\end{equation}
    By Claim \ref{claim:Chebyshev sum inequality} and Equation \eqref{eq:density average}, 
	$$
	e(U) \cdot d_X + e(V) \cdot d_Y 
	\geq 
	(e(U)+e(V)) \cdot \frac{d_X + d_Y}{2} \geq (e(F) - e(U,V)) \cdot \left( 2p - d_{X,Y} - \frac{1.5}{n} \right).
	$$
	Plugging this into \eqref{eq:expected number of c-color edges in random embedding}, we get
    \begin{align}\label{eq:expected number of c-color edges in random embedding 2}
	\mathbb{E}[|E(F_0) \cap E(G)|] \nonumber &\geq
	(e(F) - e(U,V)) \cdot \left( 2p - d_{X,Y} - \frac{1.5}{n} \right) + e(U,V) \cdot d_{X,Y}
	\\ &\geq 
	p \cdot e(F) + \left( e(F) - 2e(U,V) \right) \cdot (p - d_{X,Y}) - \frac{1.5e(F)}{n}.
	\end{align}
    We may assume that if 
    $e(U,V) \leq \frac{e(F)}{2} - t$ then
    $d_{X,Y} \leq \frac{1}{2} - 4\gamma$, and if $e(U,V) \geq \frac{e(F)}{2} + t$ then
    $d_{X,Y} \geq \frac{1}{2} + 4\gamma$. 
    Indeed, this can be guaranteed by replacing $G$ with its complement $\overline{G}$, if necessary, and noting that if the conclusion of Lemma \ref{lem:cut discrepancy} holds for $\overline{G}$ then it also holds for $G$. 
    Recall also that $|p - \frac{1}{2}| \leq \frac{\gamma t}{e(F)} \leq \gamma$, and so $|d_{X,Y} - p| \geq 3\gamma$.
    It now follows that
    $$
    \left( e(F) - 2e(U,V) \right) \cdot (p - d_{X,Y}) \geq 
    6\gamma t.
    $$
    Plugging this into \eqref{eq:expected number of c-color edges in random embedding 2} and using that $p \geq \frac{1}{2} - \frac{\gamma t}{e(F)}$, we get that
    $$
    \mathbb{E}[|E(F_0) \cap E(G)|] \geq 
    \frac{e(F)}{2} + 5\gamma t - 
    \frac{1.5e(F)}{n} \geq 
    \frac{e(F)}{2} + 4\gamma t.
    $$
	Hence, there is a choice for $f$ such that 
    $|E(F_0) \cap E(G)| \geq \frac{e(F)}{2} + \gamma t$, as required.
\end{proof}

As mentioned above, we will combine Lemma \ref{lem:cut discrepancy} with known results on the existence of unbalanced bisections. 
Lee, Loh and Sudakov \cite{LLS} studied bisections in graphs of given maximum degree, proving that every $n$-vertex graph $F$ with maximum degree $\Delta$ has a bisection of size at least $\frac{e(F)}{2} + \frac{n-\max(\tau,\Delta-1)}{4}$, where $\tau$ is the number of {\em tight components} of $F$ (see \cite[Theorem 1.3]{LLS}). We will not need to define this notion precisely; it suffices for us that every tight component has an odd number of vertices. Therefore, if $F$ has no isolated vertices then $\tau \leq \frac{n}{3}$. We thus get the following: 
\begin{theorem}[\cite{LLS}]\label{thm:bisection max degree}
Every $n$-vertex graph $F$ with maximum degree $\Delta$ and no isolated vertices has a partition $V(F) = U \cup V$ with $|U| = \lfloor \frac{n}{2} \rfloor$ and $e(U,V) \geq \frac{e(F)}{2} + \min \left( \frac{n}{6},\frac{n+1-\Delta}{4} \right)$.\footnote{We note that when using Theorem \ref{thm:bisection max degree} in the proof of Theorem \ref{thm:max degree}, we will only apply Theorem \ref{thm:bisection max degree} in a case where the maximum degree of $F$ is assumed to be small, i.e., at most $\delta n$ for some small constant $\delta$. Hence, the conclusion is that there is a bisection of size at least $\frac{m}{2}+\frac{n}{6}$.}
\end{theorem}

Moving on to $d$-regular graphs, we need some definitions. For an $n$-vertex graph $F$ with density $p := e(F)/\binom{n}{2}$ and for a subset $U \subseteq V(F)$, denote
$\disc(U) := e(U) - p\binom{|U|}{2}$. Also, $\disc^+(F)$ denotes the maximum of $\disc(U)$ over all subsets $U \subseteq V(F)$, and $\disc^-(F)$ denotes the maximum of $-\disc(U)$ over all subsets $U \subseteq V(F)$. 
Bollob\'as and Scott \cite{Bollobas_Scott} proved that if $p(1-p) \geq 1/n$, then $\disc^+(F) \cdot \disc^-(F) \geq \Omega(p(1-p)n^3)$. In particular, if $F$ is $d$-regular with $2 \leq d \leq (1-\varepsilon)n$ and $n \geq 2/\varepsilon$, then $\disc^+(F) \cdot \disc^-(F) \geq \Omega(\varepsilon dn^2)$. Therefore,
$\max\{ \disc^+(F), \disc^-(F) \} \geq \Omega(\sqrt{\varepsilon d} n)$.
R\"ati, Sudakov and Tomon \cite[Lemma 2.6]{RST} proved that for a $d$-regular $F$ on $n \geq 100$ vertices, there exists a bisection $V(F) = U \cup V$ with $e(U,V) \leq \frac{e(F)}{2} - \frac{1}{3}\disc^+(F)$, and there exists a bisection $V(F) = U \cup V$ with $e(U,V) \geq \frac{e(F)}{2} + \frac{1}{3}\disc^-(F)$; this second statement does not appear explicitly in \cite{RST} (they instead state it for general partitions $U,V$, rather than for bisections), but it is easy to see that the proof in \cite{RST} gives this statement.   
Combining all of the above, we get:
\begin{theorem}[Follows from \cite{Bollobas_Scott,RST}]\label{thm:bisection d-regular}
There is an absolute constant $c = c_{\ref{thm:bisection d-regular}} > 0$ such that the following holds for every $\varepsilon > 0$, every $n \geq n_0(\varepsilon)$, and every $d$-regular $n$-vertex graph $F$ with $2 \leq d \leq (1-\varepsilon)n$. There exists a partition $V(F) = U \cup V$ with $|U| = \lfloor \frac{n}{2} \rfloor$ such that $\big| e(U,V) - \frac{e(F)}{2} \big| \geq c \cdot \sqrt{\varepsilon d}n$. 
\end{theorem}

\subsection{The main lemma}\label{subsec:main lemma}
In this section, we state and prove Lemma \ref{lem:main}, which is one of the main ingredients in our proofs.
Before stating this lemma, we need the following two definitions, introducing (somewhat technical) conditions on the guest graph $F$ and the host graph $G$. These conditions will allow us to find a copy of $F$ with high discrepancy.

\begin{definition}[guest-good]\label{def:guest}
An $n$-vertex graph $F$ is {\em guest-good with value $t$} if there is a partition $V(F) = U \cup V$ with $|U| = \lfloor \frac{n}{2} \rfloor$, and there are distinct vertices $u_1,u'_1,\dots,u_m,u'_m \in U$, $m \leq 0.05n$, and integers $d_1,\dots,d_m \geq 1$, such that $t=\sum_{i=1}^m \sqrt{d_i}$ and, moreover, the following hold for every \nolinebreak $i \in [m]$.
\begin{enumerate}
    \item $|N_V(u_i) \setminus N_V(u'_i)| \geq 0.01d_i$.
    \item $|N_V(u_i) \setminus N_V(u'_i)|,
    |N_V(u'_i) \setminus N_V(u_i)| \leq \frac{2}{3}|V|$.
    \item $U$ is independent, or $|d_V(u_i) - d_V(u'_i)| \leq 20\sqrt{d_i}$. 
\end{enumerate}
\end{definition}

\begin{definition}[host-$\beta$-good]\label{def:host}
An $n$-vertex graph $G$ is {\em host-$\beta$-good} if there is a partition $V(G) = X \cup Y$ with $|X| = \lfloor \frac{n}{2} \rfloor$, and there are distinct vertices $x_1,x'_1,\dots,x_m,x'_m \in X$, $m = 0.05n$, such that for every $1\leq i\leq m,$ we have 
\begin{enumerate}
    \item $|d_Y(x_i) - d_Y(x'_i)| \leq \beta n$, 
    \item $0.02|Y| \leq |N_Y(x_i) \triangle N_Y(x'_i)| \leq 0.98|Y|$, and
    \item  $0.1|Y| \leq d_Y(x_i),d_Y(x'_i) \leq 0.9|Y|$.
\end{enumerate}
\end{definition}

\noindent
We are now ready to state the main result of this section.
\begin{lemma}\label{lem:main}
    There is an absolute constant 
    $\beta = \beta_{\ref{lem:main}} > 0$ such that the following holds for all large enough $n$. 
    Let $F$ and $G \subseteq K_n$ be $n$-vertex graphs, and suppose that $F$ is guest-good with value $t$ and $G$ is host-$\beta$-good. Then there is a copy $F_0$ of $F$ in $K_n$ with
    $\big| |E(F_0) \cap E(G)| - \frac{e(F)}{2} \big| \geq \beta t$. 
\end{lemma}

For the proof of Lemma \ref{lem:main}, we need the probabilistic Lemma \ref{lem:sqrt deviation} below. For this lemma we in turn need the following four Lemmas \ref{lem:coupling}-\ref{lem: random matching}.
\begin{lemma}\label{lem:coupling}
Let $P,Q \subseteq [n]$ be disjoint subsets, let $1 \leq k \leq n$, and let $A \subseteq [n]$ be chosen uniformly at random among all sets of size $k$. Then
\begin{enumerate}
    \item If $|P| \geq |Q|$ then
    $\mathbb{P}\left[ |A \cap P| \geq |A \cap Q| \right] \geq \frac{1}{2}$.
    \item For all $s,t$, the events $|A\cap P|\geq s$ and $|A\cap Q|\leq t$ are positively correlated.
\end{enumerate}
\end{lemma}
\begin{proof}
    For Item 1, fix $P' \subseteq P$ with $|P'| = |Q|$. It is enough to show that $|A \cap P'| \geq |A \cap Q|$ with probability at least $\frac{1}{2}$. But by symmetry, the events $\mathcal{A} = \{|A \cap P'| \geq |A \cap Q|\}$
    and 
    $\mathcal{B} = \{|A \cap P'| \leq |A \cap Q|\}$ have the same probability, and $\mathbb{P}[\mathcal{A}] + \mathbb{P}[\mathcal{B}] \geq 1$. The claim follows. 

    For Item 2, we need to show that 
    $\mathbb{P}[|A\cap Q|\leq t\mid |A \cap P|\geq s]\geq \mathbb{P}[|A \cap Q|\leq t]$. Let us consider the selection of $A$ from the following perspective. First, we fix a uniformly-at-random ordering ${v_1,\ldots,v_{n-|P|}}$ of $[n]\setminus P$. Then, for $i=1,2,\ldots$, let $A_i\subseteq [n]$ be a uniformly-at-random set of size $k$ and set $s_i=|A_i\cap P|$. Let $A_i'=(A_i\cap P)\cup\{v_1,\ldots,v_{k-s_i}\}$ (so that $|A_i'| = k)$. Note that the distribution of $A_i'$ is the same as the distribution of $A$, that is, it is a uniformly-at-random subset of $[n]$ of size $k$. Let $j$ be minimal such that $s_j\geq s$. Then, $A_j'$ follows the same distribution as $A$ conditioned on $|A \cap P|\geq s$. Note that $s_j\geq s_1$, and that 
    $A_j' \cap Q = \{v_1,\dots,v_{k-s_j}\} \cap Q$ and $A_1' \cap Q = \{v_1,\dots,v_{k-s_1}\} \cap Q$. 
    Therefore, $A_j'\cap Q\subseteq A_1'\cap Q$, implying that 
    $\mathbb{P}[|A_j'\cap Q|\leq t]\geq 
    \mathbb{P}[|A_1'\cap Q|\leq t]$. Since 
    $A_j'$ follows the same distribution as $A$ conditioned on $|A \cap P|\geq s$, while $A_1'$ follows the same distribution as $A$, we arrive at the desired inequality.
\end{proof}
\noindent
The next lemma includes well-known facts about the binomial distribution.
\begin{lemma}\label{lem:binomial}
Let $n \geq 1$ and let $Z \sim \Bin(n,\frac{1}{2})$.
\begin{enumerate}
    \item $\mathbb{P}[Z \geq \frac{n}{2}] \geq \frac{1}{2}$.
    \item For every $\alpha > 0$ there is $\beta > 0$ such that if $n$ is large enough, then $\mathbb{P}\left[ Z \geq \frac{n}{2} + \beta \sqrt{n} \right] \geq \frac{1}{2}-\alpha$.
\end{enumerate}
\end{lemma}
\begin{proof}
    The first item follows from the fact that $\binom{n}{k} = \binom{n}{n-k}$.
    The second item follows from the central limit theorem. 
\end{proof}

The following is an anti-concentration statement for the hypergeometric distribution. We suspect that such a result is also known (perhaps even in a stronger form), but could not find a reference, and hence give a proof. As the proof is via a somewhat tedious calculation, we defer it to the appendix. 
\begin{lemma}\label{lem:hypergeometric}
For every $\eta > 0$, there exists $k_0 = k_0(\eta)$ such that the following holds. 
Let $k \geq k_0$ and $n$ be integers with $k \leq (1-\eta)n$, let $P \subseteq [n]$, put $p = |P|/n$, and suppose that $\eta \leq p \leq 1-\eta$. Let $A$ be a subset of $[n]$ of size $k$ chosen uniformly at random. 
Then it holds that
$\mathbb{P}[|A \cap P| \geq pk + 0.1\eta \sqrt{k}] \geq 0.04\eta$ and 
$
\mathbb{P}[|A \cap P| \leq pk - 0.1\eta \sqrt{k}] \geq 0.04\eta. 
$
\end{lemma}

\noindent
Finally, the following is the last auxiliary lemma we need for the proof of Lemma \ref{lem:sqrt deviation} below.
\begin{lemma}\label{lem: random matching}
    For every $\eta>0$, there exists $k = k_0(\eta)$ such that for all $k \geq k_0$ and $n \geq 2k$, the following holds. Let $P\subseteq [n]$ with $\eta\leq |P|/n\leq 1-\eta$ and let $M$ be a matching of size $k$ in $[n]$, chosen uniformly at random. With probability at least $\frac{5}{6}$, the number of edges $uv\in M$ with $u\in P$ and $v\notin P$ is at least $\eta k/25$.
\end{lemma}
\begin{proof}
Note that we may assume $\eta\leq \frac 12$ and, without loss of generality, $|P|\leq n/2$. Let $X,Y\subseteq [n]$ be disjoint sets of size $k$ chosen uniformly at random. Let $M$ be the matching between $X$ and $Y$ obtained by the following process. Let $X=\{x_1,\ldots,x_k\}$, where we order the vertices such that $P\cap X=\{x_1,\ldots,x_{|P\cap X|}\}$. Then, for $1\leq i\leq k$, let $y_i$ be chosen uniformly at random from $Y\setminus\{y_1,\ldots,y_{i-1}\}$ and add $x_iy_i$ to $M$. Note that $M$ is a perfect matching between $X$ and $Y$ chosen uniformly at random. Since $X$ and $Y$ are chosen uniformly at random themselves, it follows that $M$ is chosen uniformly at random among all matchings of size $k$ in $[n]$.

Using standard results on the hypergeometric distribution, we get that, with probability at least $11/12$, both $|P\cap X|\geq\eta k/2$ and $|Y\setminus P|\geq k/3$. Suppose this is the case. Then, for $1\leq i\leq \eta k/2$, $y_i\notin P$ with probability at least $1/12$, as $\eta k/2\leq k/4$. We say that an edge $uv\in M$ is good if $u\in P$ and $v\notin P$. By the above argument, it follows that if $|P\cap X|\geq\eta k/2$ and $|Y\setminus P|\geq k/3$ then the number of good edges in $M$ stochastically dominates $\Bin(\eta k/2, 1/12)$. In this case, by the Chernoff bound we get that $M$ contains at least $\eta k/25$ good edges with probability at least $11/12$.

By a union bound, we conclude $M$ contains at least $\eta k/25$ good edges with probability at least $1-2\cdot1/12=5/6$. 
\end{proof}
\begin{lemma}\label{lem:sqrt deviation}
    For every $\eta > 0$ there exists 
    $\rho = \rho_{\ref{lem:sqrt deviation}}(\eta) > 0$ such that the following holds for all large enough $n$. 
    Let $P,Q \subseteq [n]$ be disjoint subsets, put $p = |P|/n, q = |Q|/n$, and suppose that $p \geq q$ and $\eta \leq p \leq 1-\eta$. 
    Let $0 \leq b \leq a\leq (1-\eta)n$, and let
    $A,B \subseteq [n]$ be disjoint sets with $|A|= a, |B| =b$, chosen uniformly at random among all such pairs of subsets of $[n]$. 
    Let $Z := |A \cap P| - |A \cap Q| - |B \cap P| + |B \cap Q|$. Then $\mathbb{P}[Z \geq \rho \sqrt{a}| \geq \rho$. 
\end{lemma}
\begin{proof}
    We assume that $a$ is large enough as a function of $\eta$. Otherwise, with some positive probability lower-bounded solely by a function of $\eta$, it holds that $A \subseteq P$ and $B \subseteq [n] \setminus P$, in which case $Z \geq a \geq \rho \sqrt{a}$. Thus, choosing $\rho$ small enough handles this case.\footnote{We note that in the regime $1 \ll a \ll \sqrt{n}$, the lemma can be proved more easily by sampling $A,B$ with replacement instead of without (as there are typically no collisions when $a \ll \sqrt{n}$) and using the central limit theorem.}

    Instead of sampling $A$ and $B$ directly, consider the following process to sample $A$ and $B$ according to the same distribution. First, let $M$ be a uniformly-at-random matching on $[n]$ of size $b$ and let $A'\subseteq [n]$ be a 
    uniformly-at-random set of size $a-b$ disjoint from $V(M)$. Then, for each edge $e = uv\in M$, let $a_e$ be a uniformly-at-random vertex of $\{u,v\}$. Let $A_M=\{a_e : e\in M\}$ and $B_M=V(M)\setminus A_M$. Finally, set $A=A'\cup A_M$ and $B=B_M$. Note that $A'$ is distributed as a uniformly-random subset of $[n]$ of size $a-b$. 
    
    Later, we distinguish two cases. In both of these we use the following notation and observations. Let $Z':=|A'\cap P|-|A'\cap Q|$ and $Z_M:=|A_M \cap P| - |A_M \cap Q| - |B_M \cap P| + |B_M \cap Q|$ and note that $Z=Z'+Z_M$. Let us partition $M$ into $M_0,M_1,M_2$ and $M_2'$, where $M_0$ contains all the edges with neither endpoint in $P\cup Q$; $M_1$ contains the edges with exactly one endpoint in $P\cup Q$; $M_2$ contains the edges with one endpoint in $P$ and one in $Q$; and $M_2'$ contains the edges with either both endpoints in $P$ or both in $Q$. For $e\in M_0\cup M_2'$, the choice of $a_e$ does not change $Z_M$. For $e\in M_1$ (resp.~$e \in M_2$), the choice of $a_e$ changes $Z_M$ by $2$ (resp.~$4$). So, $Z_M$ is distributed as $2X+4Y$, where $X\sim \Bin(|M_1|,\frac{1}{2})-\frac{1}{2}|M_1|$ and $Y\sim\Bin(|M_2|,\frac{1}{2})-\frac{1}{2}|M_2|$.
    We now move on to the case distinction.
    
    \paragraph{Case 1:} $a\leq 2b$. Let $m$ denote the number of edges in $M$ which intersect $P$ in exactly one vertex. 
    Note that $m \leq |M_1| + |M_2|$.
    As $M$ is a 
    uniformly-at-random matching of size $b$ on $[n]$, by Lemma~\ref{lem: random matching}, 
    $m\geq \eta b/25\geq \eta a/50 =: t$ with probability at least $\frac{5}{6}$. 
    We may apply Lemma~\ref{lem: random matching} because $b \geq a/2$ and we assumed that $a$ is large enough as a function of $\eta$. 
    As $A'$ is a subset of $[n]$ of size $a-b$ chosen uniformly at random, and $p\geq q$, with probability at least $\frac{1}{2}$ we have $Z'\geq 0$ (by Item 1 of Lemma \ref{lem:coupling}). Thus, with probability at least $\frac{1}{3}$, both $m\geq t$ and $Z'\geq 0$ hold. Suppose this is the case. As 
    $m\geq t$, it follows that either $|M_1|\geq t/2$ or $|M_2|\geq t/2$. 
    Therefore, by Lemma \ref{lem:binomial}, we have $Z_M=2X+4Y\geq \rho\sqrt{a}$ with probability at least $\frac{1}{5}$, say. Indeed, assuming that $|M_1| \geq t/2$ (the case $|M_2| \geq t/2$ is identical), we have
    $X \geq \rho \sqrt{a}$ with probability at least $0.49$, say (by Item 2 of Lemma \ref{lem:binomial}), and $Y \geq 0$ with probability at least $0.5$ (by Item 1 of Lemma \ref{lem:binomial}). 
    Altogether, we get that $m\geq t$ and $Z'\geq 0$ with probability at least $\frac{1}{3}$, and that conditioned on this, $Z_M\geq \rho\sqrt{a}$ with probability at least $\frac{1}{5}$. Hence, with probability at least $\frac{1}{15}$ we have $Z=Z'+Z_M\geq \rho\sqrt{a}$.

    \paragraph{Case 2:} $a>2b$. So $|A'| = a-b > \frac{a}{2}$ is assumed to be large enough as a function of $\eta$. 
    By Item 2 of Lemma \ref{lem:coupling}, we get 
    $$
    \mathbb{P}[|A'\cap P|\geq pa+\rho\sqrt{a} \text{ and } |A'\cap Q|\leq pa]\geq
    \mathbb{P}[|A'\cap P|\geq pa+\rho\sqrt{a}]\cdot
    \mathbb{P}[|A'\cap Q|\leq pa].\footnote{The event $\mathbb{P}[|A'\cap Q|\leq pa]$ may seem strange, since the expected value of $|A' \cap Q|$ is $qa$ and not $pa$. The reason we are considering this event is that we would like to invoke Lemma \ref{lem:hypergeometric} and avoid proving additional statements for the hypergeometric distribution. We cannot invoke this lemma for $|A' \cap Q|$ because $Q$ may be too small. However, $|A' \cap Q|$ is stochastically dominated by $|A' \cap P|$, and we may invoke the lemma for $|A' \cap P|$.}
    $$
    By Lemma \ref{lem:hypergeometric}, we have 
    $\mathbb{P}[|A'\cap P|\geq pa+\rho\sqrt{a}] \geq \mathbb{P}\left[ |A' \cap P| \geq 0.1\eta \sqrt{a-b} \right] \geq 0.04\eta$. 
    Also, as $|P| \geq |Q|$, we have 
    $\mathbb{P}[|A' \cap Q| \leq pa] \geq 
    \mathbb{P}[|A' \cap P| \leq pa] \geq 0.04\eta$, again using Lemma \ref{lem:hypergeometric}. Altogether, with probability at least $\frac{1}{625}\eta^2$, we have $|A' \cap P| \geq pa + \rho \sqrt{a}$ and $|A' \cap Q| \leq pa$, in which case 
    $Z' \geq \rho\sqrt{a}$. 
    Also, recall that $Z_M=2X+4Y$ and that
    $\mathbb{P}[X\geq 0],\mathbb{P}[Y\geq 0]\geq \frac{1}{2}$ (by Item 1 of Lemma \ref{lem:binomial}), no matter the choice of $M$. Hence,  
    we get that $\mathbb{P}[Z_M\geq 0]\geq \frac{1}{4}$. Altogether, $Z=Z'+Z_M\geq \rho \sqrt{a}$ with probability at least $\frac{1}{2500}\eta^2 \geq \rho$. 
\end{proof} 

\noindent
We are now ready to prove Lemma \ref{lem:main}.

\begin{proof}[Proof of Lemma \ref{lem:main}]
As $F$ is guest-good with value $t$, there is a partition $V(F) = U \cup V$ with $|U| = \lfloor \frac{n}{2} \rfloor$, and there are distinct vertices $u_1,u'_1,\dots,u_m,u'_m \in U$, $m \leq 0.05n$, and integers $d_1,\dots,d_m \geq 1$ with $\sum_{i=1}^m \sqrt{d_i} = t$, such that 
$|N_V(u_i) \setminus N_V(u'_i)| \geq 0.01d_i$
and 
$|N_V(u_i) \setminus N_V(u'_i)|, 
|N_V(u'_i) \setminus N_V(u_i)| \leq 
\frac{2}{3}|V|$ 
for every $i \in [m]$, and such that either $U$ is independent in $F$ or $|d_V(u_i) - d_V(u'_i)| \leq 20\sqrt{d_i}$ for every $i \in [m]$. 
For convenience, put $V_i := N_V(u_i) \setminus N_V(u'_i)$ and $V'_i := N_V(u'_i) \setminus N_V(u_i)$, so that $|V_i| \geq 0.01 d_i$ and $|V_i|,|V'_i| \leq \frac{2}{3}|V|$. Without loss of generality, we may assume that $|V_i| \geq |V'_i|$ (else switch $u_i$ and $u'_i$). 

As $G$ is host-$\beta$-good, there is a partition $V(G) = X \cup Y$ with $|X| = \lfloor \frac{n}{2} \rfloor$, and there are distinct vertices $x_1,x'_1,\dots,x_m,x'_m \in X$ such that for every $i \in [m]$ we have 
$|d_Y(x_i) - d_Y(x'_i)| \leq \beta n$, 
$0.02|Y| \leq |N_Y(x_i) \triangle N_Y(x'_i)| \leq 0.98|Y|$, and $0.1|Y| \leq d_Y(x_i),d_Y(x'_i) \leq 0.9|Y|$. 
Similarly to the above, put $Y_i := N_Y(x_i) \setminus N_Y(x'_i)$ and 
$Y'_i := N_Y(x'_i) \setminus N_Y(x_i)$, and assume without loss of generality that $|Y_i| \geq |Y'_i|$ (else switch $x_i$ and $x'_i$); so 
$|Y_i| \geq 0.01|Y|$.

Take a uniformly at random bijection $g : V \rightarrow Y$. 
For $1 \leq i \leq m$, let
\begin{equation}\label{eq:good pair}
\begin{split}
	D_i &:= |g(V_i) \cap N_Y(x_i)| - |g(V_i) \cap N_Y(x'_i)| - |g(V'_i) \cap N_Y(x_i)| + |g(V'_i) \cap N_Y(x'_i)| \\ \\ &=  
    |g(V_i) \cap Y_i| - |g(V_i) \cap Y'_i| - 
    |g(V'_i) \cap Y_i| + |g(V'_i) \cap Y'_i|.
\end{split}
\end{equation}
Let $\rho = \rho_{\ref{lem:sqrt deviation}}(0.01)$ be given by Lemma \ref{lem:sqrt deviation} with parameter $\eta = 0.01$.
We say that the pair $u_iu'_i$ is {\em good} if 
$D_i \geq 0.1\rho \sqrt{d_i}$. 

\begin{claim}\label{claim:good pair probability}
    For every $1 \leq i \leq m$, 
    $\mathbb{P}[u_iu'_i \text{ good}] \geq \rho$.    
\end{claim}
\begin{proof}
    We apply Lemma \ref{lem:sqrt deviation} with $Y$ in place of $[n]$ and with $P := Y_i$, $Q := Y'_i$. Indeed, note that $(A,B) := (g(V_i),g(V'_i))$ is distributed as a random pair of subsets of $Y$ of sizes $|A| = a := |V_i|$ and $|B| = b := |V'_i|$. Also, setting $p = |P|/|Y|$ and $q = |Q|/|Y|$, we have $p \geq 0.01$ because $|Y_i| \geq 0.01|Y|$; 
    $p \leq 0.9$ because $|Y_i| \leq |N_Y(x_i)| \leq 0.9|Y|$; and $p \geq q$ because $|Y_i| \geq |Y'_i|$.
    Also, $a = |V_i| \geq |V'_i| = b$; $a \geq 0.01d_i$; and 
    $a \leq \frac{2}{3}|V| = \frac{2}{3}|Y|$.
    Note that the expression $|A \cap P| - |A \cap Q| - |B \cap P| + |B \cap Q|$ in Lemma \ref{lem:sqrt deviation} is exactly the bottom line in \eqref{eq:good pair}. 
    Hence,
    by Lemma \ref{lem:sqrt deviation}, 
    $\mathbb{P}[D_i \geq 0.1\rho \sqrt{d_i}] \geq 
    \mathbb{P}[D_i \geq \rho\sqrt{a}] \geq \rho$.
    This proves the claim.
\end{proof}

Let $Z$ be the random variable
$$
Z := \sum_{i=1}^m D_i \cdot \mathds{1}_{u_iu'_i \text{ is good}}
$$
By Claim \ref{claim:good pair probability} and linearity of expectation, we have 
$\mathbb{E}[Z] \geq 
\sum_{i=1}^m 0.1\rho^2 \sqrt{d_i} = 0.1\rho^2 t \geq 3\sqrt{\beta}t$, where the last inequality holds if $\beta$ is small enough. 
Hence, there is an outcome for $g$ 
such that $Z \geq 3\sqrt{\beta} t$. Fix such an outcome $g$, and suppose without loss of generality that for this $g$, the pair $u_iu'_i$ is good if and only if $i \leq k$ (for some $k$). So 
\begin{equation}\label{eq:contribution of good pairs}
Z = \sum_{i=1}^k D_i \geq 3\sqrt{\beta} t.
\end{equation}

Fix a bijection $h_1 : U \rightarrow X$ satisfying 
$h_1(u_i) = x_i$ and $h_1(u'_i) = x'_i$ for every $1 \leq i \leq k$. 
Let $\sigma : U \rightarrow U$ be the permutation of $U$ satisfying $\sigma(u_i) = u'_i, \sigma(u'_i) = u_i$ for every $1 \leq i \leq k$ and $\sigma(w) = w$ for every $w \in U \setminus \{u_i,u'_i : 1 \leq i \leq k\}$. 
In other words, $\sigma$ switches $u_i,u'_i$ for $1 \leq i \leq k$, and keeps all other points of $U$ in place. Set $h_2 := h_1 \circ \sigma$; so $h_2$ is also a bijection from $U$ to $X$. 
Note that $h_1(u_i) = x_i, h_1(u'_i) = x'_i$ and $h_2(u_i) = x'_i, h_2(u'_i) = x_i$ for $1 \leq i \leq k$, and that $h_1(w) = h_2(w)$ for all 
$w \in U \setminus \{u_i,u'_i : 1 \leq i \leq k\}$.  
Let $H := F[U]$. 
For $i=1,2$, let $H_i$ be the copy of $H$ in $K_n$ given by $h_i$; namely $H_i = \{h_i(e) : e \in E(H)\}$.
We consider two cases.

\paragraph{Case 1:} $|E(H_1) \cap E(G)| - |E(H_2) \cap E(G)| \geq - 2\sqrt{\beta} t$.
	In this case, we consider the following two copies of $F$ in $K_n$. 
	For $i = 1,2$, let $f_i := h_i \cup g$; so $f_i : V(F) \rightarrow V(K_n)$ is a bijection. Let $F_i$ be the copy of $F$ given by $f_i$, i.e., $F_i = \{f_i(e) : e \in E(F)\}$. Note that $H_i \subseteq F_i$. We claim that
	\begin{equation}\label{eq:F_1,F_2. Case 1}
	|E(F_1) \cap E(G)| - |E(F_2) \cap E(G)| = 
    |E(H_1) \cap E(G)| - |E(H_2) \cap E(G)| + \sum_{i=1}^k D_i.
	\end{equation}
	Indeed, let us consider the contribution of different types of edges to 
    $S := |E(F_1) \cap E(G)| - |E(F_2) \cap E(G)|$. As $F_1,F_2$ contain the same edges inside $Y$, these edges do not contribute to $S$. Also, the contribution of the edges contained in $X$ is $|E(H_1) \cap E(G)| - |E(H_2) \cap E(G)|$. Let us now consider the edges between $X$ and $Y$. For an edge $e \in E(F)$ with one end in $V$ and one in $U \setminus \{u_iu'_i : 1 \leq i \leq k\}$, we have $f_1(e) = f_2(e)$ (so $e$ is in both $F_1$ and $F_2$), hence these edges again do not contribute to $S$. Finally, let $1 \leq i \leq k$, and consider the edges between $x_i,x'_i$ and $Y$. For $v \in V$, if 
    $v \in N_V(u_i) \cap N_V(u'_i)$ then 
    $x_ig(v), x'_ig(v)$ are edges of both $F_1$ and $F_2$, again giving no contribution to $S$. If 
    $v \in N_V(u_i) \setminus N_V(u'_i) = V_i$, then 
    $x_ig(v) \in E(F_1)$ while $x'_ig(v) \in E(F_2)$. Hence, $v$ contributes 
    $\mathds{1}_{g(v) \in N_Y(x_i)} - 
    \mathds{1}_{g(v) \in N_Y(x'_i)}$ to 
    $S$. This means that the vertices in $V_i$ contribute
	$|g(V_i) \cap N_Y(x_i)| - |g(V_i) \cap N_Y(x'_i)|$ to $S$. Similarly, the vertices of $V'_i$ contribute 
	$|g(V'_i) \cap N_Y(x'_i)| - |g(V'_i) \cap N_Y(x_i)|$ to $S$. Thus, the contribution of the edges between $u_i,u'_i$ and $V$ is 
	$D_i$ (see Equation \ref{eq:good pair}). This proves \eqref{eq:F_1,F_2. Case 1}. 
    Now, by \eqref{eq:F_1,F_2. Case 1} and \eqref{eq:contribution of good pairs}, we have 
    $$|E(F_1) \cap E(G)| - |E(F_2) \cap E(G)| = 
    |E(H_1) \cap E(G)| - |E(H_2) \cap E(G)| + Z \geq -2\sqrt{\beta} t + 3 \sqrt{\beta} t \geq 2\beta t,$$ assuming that $\beta$ is small enough. Here, the first inequality also uses the assumption of Case 1. 
    As $|E(F_1) \cap E(G)| - |E(F_2) \cap E(G)| \geq 2\beta t$, there exists $i \in \{1,2\}$ with 
    $||E(F_i) \cap E(G)| - \frac{e(F)}{2}| \geq \beta t$, as required. This completes the proof in Case 1.

    \paragraph{Case 2:}
	$|E(H_1) \cap E(G)| - |E(H_2) \cap E(G)| \leq -2\sqrt{\beta} t$. 
    This in particular means that $U$ is not independent in $F$, because else we would have $E(H) = \emptyset$ and, hence, $E(H_1),E(H_2) = \emptyset$.
    By the definition of guest-goodness, we have 
    $|d_V(u_i) - d_V(u'_i)| \leq 20\sqrt{d_i}$ for all $i \in [m]$.
    Now consider the following quantity:
	\begin{equation}
	T := \frac{1}{|Y|}\sum_{i=1}^k \left( d_V(u_i) - d_V(u'_i) \right) \cdot \left( d_Y(x_i) - d_Y(x'_i) \right).
	\end{equation}
    where the degrees $d_V(u_i),d_V(u'_i)$ are in $F$, while $d_Y(x_i),d_Y(x'_i)$ are in $G$.
	For every $1 \leq i \leq k$, we have 
    $|d_V(u_i) - d_V(u'_i)| \leq 20\sqrt{d_i}$ and
    $|d_Y(x_i) - d_Y(y_i)| \leq \beta n$.
    Hence, 
	\begin{equation}\label{eq:T is small}
	T \leq \frac{1}{|Y|} \cdot \beta n \cdot \sum_{i=1}^k 20\sqrt{d_i} \leq 
    40\beta \cdot \sum_{i=1}^m \sqrt{d_i} 
    = 40\beta t
    \leq \sqrt{\beta} t. 
	\end{equation}
	
	Consider a second uniformly at random bijection $g' : V \rightarrow Y$. For $i=1,2$, let $f_i = h_i\cup g'$ and let $F_i$ be the copy of $F$ given by $f_i$, i.e., $F_i = \{f_i(e) : e \in E(F_i)\}$. We claim that
	\begin{equation}\label{eq:F_1,F_2. Case 2}
	\mathbb{E}[|E(F_1) \cap E(G)| - |E(F_2) \cap E(G)|] = 
    |E(H_1) \cap E(G)| - |E(H_2) \cap E(G)| + T.
	\end{equation}
	To see this, we again consider the contribution of different types of edges to the expression $S := |E(F_1) \cap E(G)| - |E(F_2) \cap E(G)|$. As $F_1,F_2$ have the same edges inside $Y$, these edges do not contribute to $S$. Also, the edges inside $X$ contribute $|E(H_1) \cap E(G)| - |E(H_2) \cap E(G)|$. Next we consider the edges between $X$ and $Y$. For an edge $e \in E(F)$ with one endpoint in $V$ and one in 
    $U \setminus \{u_i,u'_i : 1 \leq i \leq k\}$, we have $f_1(e) = f_2(e)$, so such edges belong to both $F_1$ and $F_2$ and hence do not contribute to $S$. Next, consider edges between $V$ and $u_i,u'_i$ for $1 \leq i \leq k$. For an edge $u_iv$ with $v \in V$, the probability that 
	$x_ig'(v) = f_1(u_iv)$ belongs to $G$ is 
    $d_Y(x_i)/|Y|$, as $g'(v)$ is a uniformly random element of $Y$. Similarly, the probability that 
    $x'_ig'(v) = f_2(u_iv)$ belongs to $G$ is 
    $d_Y(x'_i)/|Y|$. Hence, by linearity of expectation, the edges between $u_i$ and $V$ contribute 
	$
	d_V(u_i) \cdot \frac{1}{|Y|} \left( d_Y(x_i) - d_Y(x'_i) \right)
	$ 
	to $\mathbb{E}[S]$. 
	Similarly, the edges between $u'_i$ and $V$ contribute 
	$
	d_V(u'_i) \cdot \frac{1}{|Y|} \left( d_Y(x'_i) - d_Y(x_i) \right)
	$ 
	to $\mathbb{E}[S]$. Summing over all $i \in [k]$ proves \eqref{eq:F_1,F_2. Case 2}. 
	
	By the assumption of Case 2, we have 
    $|E(H_1) \cap E(G)| - |E(H_2) \cap E(G)| \leq -2\sqrt{\beta} t$, and by \eqref{eq:T is small} we have 
    $T \leq \sqrt{\beta} t$. Hence, by \eqref{eq:F_1,F_2. Case 2}, 
    $\mathbb{E}[|E(F_1) \cap E(G)| - |E(F_2) \cap E(G)|] \leq -\sqrt{\beta}t \leq -2\beta t$. Fixing an outcome of $g'$ with $|E(F_1) \cap E(G)| - |E(F_2) \cap E(G)| \leq -2\beta t$, there exists $i \in \{1,2\}$ such that
    $||E(F_i) \cap E(G)| - \frac{e(F)}{2}| \geq \beta t$, as required.
\end{proof}

\subsection{Setting up the host graph}\label{subsec:host}
Thanks to Lemma \ref{lem:cut discrepancy}, we will be able to assume that in the host graph $G$, all bisections have the ``correct" number of edges. This suffices for $G$ to be host-$\beta$-good, as shown by the following lemma.
\begin{lemma}\label{lem:host-good}
    For every $\beta > 0$ there is $\gamma = \gamma_{\ref{lem:host-good}}(\beta) > 0$ such that the following holds for all large enough $n$. Let $G$ be an $n$-vertex graph such that for every partition $V(G) = X \cup Y$ with 
    $|X| = \lfloor \frac{n}{2} \rfloor$, it holds that 
    $|e(X,Y) - \frac{1}{2}|X||Y|| \leq \gamma n^2$. Then $G$ is host-$\beta$-good.
\end{lemma}
\begin{proof}
	We may assume that $\beta$ is small enough and that $1/\beta$ is an integer.
	Take a random partition $V(G) = X_0 \cup Y_0$ by placing each vertex in one of the parts $X_0,Y_0$ with probability $\frac{1}{2}$, independently. By the Chernoff bound and a union bound over the vertices, w.h.p. it holds that $|d_{X_0}(v) - d_{Y_0}(v)| = o(n)$ for every vertex $v \in V(G)$. Moreover, w.h.p. $|X_0|,|Y_0| = (\frac{1}{2} + o(1))n$. Fix a partition $X_0,Y_0$ such that these two conditions are satisfied. Move $o(n)$ vertices between the sets $X_0,Y_0$ to get a partition $V(G) = X \cup Y$ with $|X| = \lfloor n/2 \rfloor$. We still have $|d_X(v) - d_Y(v)| = o(n)$ for every \nolinebreak vertex \nolinebreak $v$. 

    	Let $X_{\text{high}} = \{x \in X : d_Y(x) > 0.9|Y|\}$, 
	$X_{\text{low}} = \{x \in X : d_Y(x) < 0.1|Y|\}$. First we claim that
	$|X_{\text{high}}| \leq 0.25|X|$ or 
	$|X_{\text{low}}| \leq 0.25|X|$. Indeed, suppose towards a contradiction that
	$|X_{\text{high}}|,|X_{\text{low}}| \geq 0.25|X|$. 
	For every $x \in X_{\text{high}}$, we have $d_X(x) \geq d_Y(x) - o(n) \geq (0.9 - o(1))|X|$, and therefore, 
    $d_{X_{\text{low}}}(x) \geq 
    |X_{\text{low}}| - (0.1+o(1))|X| > 
    0.5|X_{\text{low}}|$, where the last inequality uses that $|X_{\text{low}}| \geq 0.25|X|$.
	It follows that $d(X_{\text{high}},X_{\text{low}}) > 0.5$. 			
	Symmetrically, for every $x \in X_{\text{low}}$ we have
	$d_{X_{\text{high}}}(x) \leq d_X(x) \leq d_Y(x) + o(n) \leq (0.1 + o(1))|X| < 0.5|X_{\text{high}}|$, which implies that
	$d(X_{\text{high}},X_{\text{low}}) < 0.5$, a contradiction.
	This proves our claim.
	
	Without loss of generality, assume that $|X_{\text{low}}| \leq 0.25|X|$; otherwise consider the complement $\overline{G}$ in place of $G$, noting that if $\overline{G}$ is host-$\beta$-good then so is $G$. Observe that $|X_{\text{high}}| < 0.6|X|$, because otherwise 
	$e(X,Y) \geq 0.6|X| \cdot 0.9|Y| = 0.54|X||Y|$, contradicting the assumption of the lemma. Now, setting 
	$X_{\text{med}} := X \setminus (X_{\text{high}} \cup X_{\text{low}})$, we have
	$|X_{\text{med}}| \geq (1 - 0.25 - 0.6)|X| \geq 0.15|X|$. By definition, for every $x \in X_{\text{med}}$ we have $0.1|Y| \leq d_Y(x) \leq 0.9|Y|$.

    Now, we establish the following property of the partition $(X,Y)$.
    \begin{claim}\label{claim:balanced partitions switching}
    For all $X' \subseteq X$ and $Y' \subseteq Y$ with $|X'| = |Y'|$, we have
    $|e(X',Y') - e(X') - e(Y')| \leq 2\gamma n^2$.
    \end{claim}
    \begin{proof}
        Let $(Z,W)$ be the partition of $V(G)$ obtained from $(X,Y)$ by moving $X'$ to $Y$ and $Y'$ to $X$; namely, $Z = (X \setminus X') \cup Y'$ and $W = (Y \setminus Y') \cup X'$. Clearly, $|Z| = |X| = \lfloor n/2 \rfloor$. Moreover, we have
		\begin{equation}\label{eq:cut property lemma, 1}
		e(Z,W) - e(X,Y) = e(X',X\setminus X') + e(Y', Y \setminus Y') - e(X',Y \setminus Y') - e(Y', X \setminus X').
		\end{equation}
		Also, using that $d_X(x) = d_Y(x) +o(n)$ for all $x\in X$,
        $$
        e(X',X\setminus X') = \sum_{x \in X'} d_X(x) - 2e(X') = \sum_{x \in X'}d_Y(x) + o(n^2) - 2e(X') = e(X',Y) + o(n^2) - 2e(X').
		$$
		Similarly,
		$e(Y',Y \setminus Y') = e(Y',X) + o(n^2) - 2e(Y')$.
		Plugging this into \eqref{eq:cut property lemma, 1}, we get
		\begin{align*}\label{eq:cut property lemma, 2}
		e(Z,W) - e(X,Y) &= \nonumber e(X',Y) - 2e(X') + e(Y',X) - 2e(Y') - e(X',Y \setminus Y') - e(Y', X \setminus X') + o(n^2)
		\\ &= 
		2(e(X',Y') - e(X') - e(Y')) + o(n^2).
		\end{align*}
		On the other hand, 
        $|e(Z,W) - e(X,Y)| \leq 2\gamma n^2$,
        since each of $e(X,Y),e(Z,W)$ is within $\gamma n^2$ from $\frac{1}{2}\lfloor \frac{n}{2} \rfloor \lceil \frac{n}{2} \rceil$, 
        by the assumption of the lemma. Hence,
        $|e(X',Y') - e(X') - e(Y')| \leq 2\gamma n^2$, as required.
    \end{proof}

    \noindent
    Next we prove:
	
	\begin{claim}
		For every subset $X^* \subseteq X_{\text{med}}$ with $|X^*| \geq \sqrt{160\gamma }n$, there exist $x,x' \in X^*$ satisfying 
        $0.02|Y| \leq |N_Y(x) \triangle N_Y(x')| \leq 0.98 |Y|$.
	\end{claim}
	
	\begin{proof}
		Suppose not. Fix $x_0 \in X^*$, and 
        set $Y_1 := N_Y(x_0), Y_2 = Y \setminus Y_1$.
        As $x_0 \in X_{\text{med}}$, we have 
        $|Y_1|,|Y_2| \geq 0.1|Y|$.
        By assumption, for every $x \in X^*$ it holds that 
        $|N_Y(x) \triangle Y_1| = |N_Y(x) \triangle N_Y(x_0)| \leq 0.02|Y|$, or 
        $|N_Y(x) \triangle Y_1| \geq 0.98|Y|$ and hence
        $|N_Y(x) \triangle Y_2| \leq 0.02|Y|$. 
        The two cases $|N_Y(x) \triangle Y_i| \leq 0.02 |Y|$ for $i=1,2$ are symmetric, so let us assume, without loss of generality, that at least half of the vertices $x \in X^*$ satisfy 
        $|N_Y(x) \triangle Y_1| \leq 0.02 |Y|$. 
        Let $X' \subseteq X^*$ be a set of $\sqrt{40\gamma}n$ such vertices (such a set $X'$ exists as $|X^*| \geq \sqrt{160\gamma}n$).
        For every $x \in X'$ we have 
        $|N_Y(x) \triangle Y_1| \leq 0.02 |Y|$ and hence $|N_Y(x) \cap Y_1| \geq |Y_1| - 0.02|Y| \geq 0.8|Y_1|$ and 
        $|N_Y(x) \cap Y_2| \leq 0.02 |Y| \leq 0.2|Y_2|$. These imply that $d(X',Y_1) \geq 0.8$ and $d(X',Y_2) \leq 0.02$, respectively. 
        By averaging, there are $Y'_i \subseteq Y_i$, $i\in\{1,2\}$, with $|Y'_i| = |X'|$ and
		$d(X',Y'_1) \geq 0.8$, $d(X',Y'_2) \leq 0.2$. 
        Hence,
		\begin{equation*}\label{eq:cut property lemma, 4}
		e(X',Y'_1) - e(X',Y'_2) \geq 0.6|X'|^2.
		\end{equation*}
		On the other hand, by Claim \ref{claim:balanced partitions switching}, we have
		\begin{equation*}\label{eq:cut property lemma, 5}
		|e(X',Y'_1) - e(Y'_1) - e(X',Y'_2) + e(Y'_2)|\leq |e(X',Y'_1) - e(X') - e(Y'_1)|+|e(X',Y'_2) - e(X') - e(Y'_2)| \leq 4\gamma n^2.
		\end{equation*}
		By the triangle inequality,
		\begin{align*}
		|e(Y'_1) - e(Y'_2)| &\geq |e(X',Y'_1) - e(X',Y'_2)| - |e(X',Y'_1) - e(Y'_1) - e(X',Y'_2) + e(Y'_2)|
		\\ &\geq
		0.6 |X'|^2 - 4\gamma n^2 > 0.5|X'|^2,
		\end{align*}
		using that $|X'| \geq \sqrt{40\gamma }n$.
		On the other hand, $e(Y'_1),e(Y'_2) \leq \binom{|X'|}{2} < \frac{1}{2}|X'|^2$ (as $|Y'_1| = |Y'_2| = |X'|$), and so $|e(Y'_1) - e(Y'_2)|  < \frac{1}{2}|X'|^2$. This contradiction completes the proof of the claim. 
	\end{proof}
    We now complete the proof of the lemma. 
    For each $0 \leq i < \frac{1}{\beta}$, let $X_i$ be the set of all vertices $x \in X_{\text{med}}$ with $i \beta n \leq d_Y(x) \leq (i+1)\beta n$. Clearly, $X_{\text{med}} = \bigcup_{0 \leq i < 1/\beta} X_i$, and for every $x,x' \in X_i$ we have $|d_Y(x) - d_Y(x')| \leq \beta n$. 
    Now, for each $0 \leq i < \frac{1}{\beta}$, we repeatedly remove from $X_i$ a pair of vertices $x,x'$ with $0.02|Y| \leq |N_Y(x) \triangle N_Y(x')| \leq 0.98|Y|$, until no such pairs remain. By the claim, we can continue this as long as there are at least $\sqrt{160\gamma }n$ remaining vertices. Hence, we extract at least $\frac{|X_i| - \sqrt{160\gamma }n}{2}$ such pairs from $X_i$. Summing over all $i$, the total number of pairs is at least
	$$
	\sum_{i=0}^{1/\beta-1} \frac{|X_i| - \sqrt{160\gamma }n}{2} = 
	\frac{|X_{\text{med}}|}{2} - \frac{1}{\beta} \cdot \frac{\sqrt{160 \gamma}n}{2} \geq 
	0.15\cdot\frac{|X|}{2} - \frac{1}{\beta} \cdot \frac{\sqrt{160 \gamma}n}{2} \geq 
	0.05n,	
	$$
	assuming that $\gamma$ is small enough as a function of $\beta$. Also, for each extracted pair $x,x'$, we have $|d_Y(x) - d_Y(x')| \leq \beta n$ because $x,x' \in X_i$ for some $i$. This shows that $G$ is host-$\beta$-good.
\end{proof}

\subsection{Putting it all together}\label{subsec:proof of theorems}

We now combine all of the above to prove Theorems \ref{thm:max degree} and \ref{thm:d-regular}.
For Theorem \ref{thm:d-regular}, we need the following simple lemma, showing that every $d$-regular graph with $d \leq \frac{n}{2}$ is guest-good with value \nolinebreak $\Omega(\sqrt{d}n)$.
\begin{lemma}\label{lem:guest-good d-regular}
    For all $n$ sufficiently large and all $d \leq \frac{n}{2}$, every $n$-vertex $d$-regular graph $F$ is guest-good with value $t \geq 0.01n\sqrt{d}$.
\end{lemma}
\begin{proof}
    We assume that $d$ is larger than some absolute constant. The case of small $d$ can be handled similarly, and we omit the details.
    Take a random partition $V(F) = U_0 \cup V_0$ by placing each vertex in one of the parts $U_0,V_0$ with probability $\frac{1}{2}$, independently.
    By Hoeffding's inequality, we have that $|U_0|,|V_0| = \frac{n}{2} \pm \sqrt{n}$ with probability at least $1 - 2/e^2 \geq 0.7$. Also, for a vertex $x \in V(F)$, we have $d_{V_0}(x) \sim \Bin(d,\frac{1}{2})$, hence $\mathbb{E}[d_{V_0}(x)] = \frac{d}{2}$ and  
	$\text{Var}[ d_{V_0}(x) ] = \frac{d}{4}$. 
    By Chebyshev's inequality we have
	$\mathbb{P}\big[ |d_{V_0}(x) - \frac{d}{2}| \geq 2\sqrt{d} \big] \leq \frac{1}{16}$. 
    Let $Z$ be the number of $u \in U_0$ with $|d_{V_0}(u) - \frac{d}{2}| \geq 2\sqrt{d}$. Then 
    $\mathbb{E}[Z] \leq \frac{n}{32}$, so 
    $\mathbb{P}[Z \leq \frac{n}{16}] \geq \frac{1}{2}$. 
    Finally, as $d \leq n/2$, we get by Hoeffding's inequality that w.h.p.~
    $d_{U_0}(x),d_{V_0}(x) \leq (\frac{1}{4} + o(1))n$ for all $x \in V(G)$.

    Fix an outcome where $|U_0|,|V_0| = \frac{n}{2} \pm \sqrt{n}$, $Z \leq \frac{n}{16}$, and 
    $d_{U_0}(x),d_{V_0}(x) \leq (\frac{1}{4} + o(1))n$ for all $x \in V(G)$. 
    Move at most $\sqrt{n}$ vertices between $U_0,V_0$ to obtain a partition $U,V$ with $|U| = \lfloor n/2 \rfloor$. The moved vertices 
    touch at most $\sqrt{n} d \leq n\sqrt{d}$ edges, hence 
    there are at most $\frac{n}{8}$ vertices $u \in U_0$ which are adjacent to at least $8\sqrt{d}$ of the moved vertices.
    Therefore, there are at least
    $|U_0| - \sqrt{n} - \frac{n}{16} - \frac{n}{8} \geq 0.3n$ vertices 
    $u \in U$ satisfying 
    $|d_V(u) - \frac{d}{2}| \leq 10\sqrt{d}$. Let $U^* \subseteq U$ be the set of these vertices; so 
    $|U^*| \geq 0.3n$. Clearly, 
    $|d_V(u) - d_V(u')| \leq 20\sqrt{d}$ and 
    $|N_V(u) \setminus N_V(u')| \leq d_V(u) \leq (\frac{1}{4} + o(1))n \leq \frac{n}{3} \leq \frac{2}{3}|V|$
    for all $u,u' \in U^*$. 
    
    Let $u_i,u'_i \in U^*$, $i = 1,\dots,m$, be a maximal collection of disjoint pairs of vertices of $U^*$ satisfying that
    $|N_V(u_i) \setminus N_V(u'_i)| \geq 0.01d$ for every $i \in [m]$. We claim that $m \geq 0.01 n$. If not, consider the set $U' := U^* \setminus \{u_i,u'_i : i \in [m]\}$, so 
    $|U'| = |U^*| - 2m \geq 0.28n$. Fix any $u \in U'$ and consider $V' := N_V(u)$. 
    Note that $|V'| \geq \frac{d}{2} - 10\sqrt{d} \geq 0.49d$, as $u \in U^*$.
    By the maximality of $m$, for every $u' \in U' \setminus \{u\}$ we have $|N_V(u') \cap V'| \geq |V'| - 0.01 d$, and therefore 
    $e(U',V') \geq |U'| \cdot \left( |V'| - 0.01 d \right)$. 
    By averaging, there exists $v \in V'$ \nolinebreak with 
    \begin{equation}\label{eq:guest-good contradiction}
    d_U(v) \geq 
    \frac{|U'|}{|V'|} \left( |V'| - 0.01 d \right) \geq 
    |U'| - \frac{0.01 d \cdot \frac{n}{2}}{|V'|} \geq 
    |U'| - \frac{0.01dn}{0.98d} \geq 0.26n.
    \end{equation}
    But $d_U(v) \leq d_{U_0}(v) + \sqrt{n} \leq (\frac{1}{4} + o(1))n$, a contradiction.
    This proves our claim that $m \geq 0.01n$. Setting $d_1 = \dots = d_m = d$, we now see that $F$ is guest-good with value 
    $t = m \sqrt{d} \geq 0.01n\sqrt{d}$. 
\end{proof}

\begin{proof}[Proof of Theorem \ref{thm:d-regular}]
Let $n$ be large enough and let $F$ be an $n$-vertex $d$-regular graph with $d \leq (1-\varepsilon)n$. 
Fix any $\{\pm 1\}$-edge-coloring of $K_n$ and let $G$ be the graph of edges of color $1$. 
We first consider the case that $d \leq \frac{n}{2}$. 
Let $c' = c_{\ref{thm:bisection d-regular}}$ be given by Theorem \ref{thm:bisection d-regular},
$\beta=\beta_{\ref{lem:main}}$ by Lemma \ref{lem:main} and $\gamma = \gamma_{\ref{lem:host-good}}(\beta)$ by Lemma \ref{lem:host-good}. We choose the parameter $c$ in the statement of the theorem to satisfy $c \ll c',\beta,\gamma$. 
We consider \nolinebreak two \nolinebreak cases.
\paragraph{Case 1:}
There exists a partition $V(G) = X \cup Y$ with $|X| = \lfloor \frac{n}{2} \rfloor$ and 
$\big| e_G(X,Y) - \frac{1}{2}|X||Y| \big| \geq \gamma n^2$. By Theorem \ref{thm:bisection d-regular}, there exists a partition $V(F) = U \cup V$ with $|U| = \lfloor \frac{n}{2} \rfloor$ and $\big| e_F(U,V) - \frac{e(F)}{2} \big| \geq c' \sqrt{d}n$. Now, by Lemma \ref{lem:cut discrepancy} with $t = c'\sqrt{d}n$, there is a copy $F_0$ of $F$ in $K_n$ with $\big|  |E(F_0) \cap E(G)| - \frac{e(F)}{2} \big| \geq 0.5\gamma t$. Thus, $F_0$ has discrepancy at least $\gamma t \geq c n \sqrt{d}$. Note that we may apply Lemma \ref{lem:cut discrepancy} because $\gamma t = \Omega(n\sqrt{d}) \gg \frac{d}{2} = \frac{e(F)}{n}$.
\paragraph{Case 2:}
Every partition $V(G) = X \cup Y$ with $|X| = \lfloor \frac{n}{2} \rfloor$ satisfies 
$\big| e_G(X,Y) - \frac{1}{2}|X||Y| \big| \leq \gamma n^2$. Then $G$ is host-$\beta$-good by Lemma \ref{lem:host-good}. 
By Lemma \ref{lem:guest-good d-regular}, $F$ is guest-good with value $t \geq \Omega(n \sqrt{d})$.
By Lemma \ref{lem:main}, there is a copy $F_0$ of $F$ in $K_n$ such that 
$\big|  |E(F_0) \cap E(G)| - \frac{e(F)}{2} \big| \geq \beta t$. Hence, $F_0$ has discrepancy at least 
$2 \beta t \geq c n \sqrt{d}$. This completes the proof. 

\vspace{0.5cm}
It remains to handle the case $\frac{n}{2} < d \leq (1-\varepsilon)n$. 
By the above, there is a constant $c > 0$ such that the given 2-coloring of $K_n$ has a copy $\overline{F_0}$ of $\overline{F}$ (the complement of $F$) with 
$\big| |E(\overline{F_0}) \cap E(G)| - \frac{e(\overline{F})}{2}\big| \geq cn\sqrt{n-1-d} \geq c\sqrt{\varepsilon d}n$. 
Let $F_0 = K_n \setminus \overline{F_0}$ be the complement of $\overline{F_0}$. 
If 
$|e(G) - 0.5\binom{n}{2}| \leq \frac{c}{2}\sqrt{\varepsilon d}n$, then  
$
\big| |E(F_0) \cap E(G)| - \frac{e(F)}{2} \big| \geq
\frac{c}{2}\sqrt{\varepsilon d}n,
$
and otherwise, 
by Lemma \ref{lem: random embedding}, we get a copy $F_1$ of $F$ with 
$\big| |E(F_1) \cap E(G)| - \frac{e(F)}{2} \big| \geq 
\frac{e(F)}{\binom{n}{2}} \cdot \frac{c}{2}\sqrt{\varepsilon d}n \geq \frac{c}{4}\sqrt{\varepsilon d}n$. This completes the proof.
\end{proof}

\begin{proof}[Proof of Theorem \ref{thm:max degree}]
Let $\varepsilon > 0$. We may and will assume that $\varepsilon$ is small enough when needed. 
Let $n$ be large enough and let $F$ be an $n$-vertex graph with maximum degree at most $(1-\varepsilon)n$ and no isolated vertices.
Let $\beta=\beta_{\ref{lem:main}}$ be given by Lemma \ref{lem:main}, and let $\gamma = \gamma_{\ref{lem:host-good}}(\beta)$ be given by Lemma \ref{lem:host-good}. 
For convenience, fix a small enough constant $\delta \ll \gamma$. We will assume that $\varepsilon \leq \delta$, and choose the parameter $c$ in the statement of the theorem to satisfy $c \ll \delta$. 

Fix any $\{\pm 1\}$-edge-coloring of $K_n$ and let $G$ be the graph of edges of color $1$. 
We may assume that $e(G)/\binom{n}{2} = 0.5 \pm 0.01$, say; else we are done by Lemma \ref{lem: random embedding}.
We consider several cases, starting with the case where $F$ has a vertex of large degree. 

\paragraph{Case 1:} $\Delta(F) \geq \delta n$.
Fix $u \in V(F)$ with $d_F(u) \geq \delta n$. 
In Case 1 we switch a single pair of vertices. We now show that there is a suitable vertex $u'$ to switch with $u$. 
		\begin{claim}\label{claim:u'}
			There is $u' \in V(F)$ with 
            $|N_F(u) \triangle N_F(u')| \geq \frac{\varepsilon}{2}n$.
		\end{claim}
		\begin{proof}
			Suppose otherwise. Put $A := N_F(u)$, $B := V(F) \setminus A$. As 
			$|N_F(u) \triangle N_F(u')| < 
            \frac{\varepsilon}{2}n$ for every $u' \in V(F)$, we have that $d_B(u') < \frac{\varepsilon}{2}n$ for every $u' \in A$, and $d_A(u') > |A| - \frac{\varepsilon}{2}n$ for every $u' \in B$. These two facts imply that $e(A,B) < |A| \cdot \frac{\varepsilon}{2}n$ and $e(A,B) > |B| \cdot (|A| - \frac{\varepsilon}{2}n) \geq (n - |A|) \cdot \frac{|A|}{2}$, using that $|A| = d_F(u) \geq \delta n \geq \varepsilon n$. Combining these two inequalities, we get that $n - |A| < \varepsilon n$ and hence $|A| > (1-\varepsilon)n$, contradicting $\Delta(F) \leq (1-\varepsilon)n$.  
		\end{proof}
		Fix a vertex $u' \in V(F) \setminus \{u\}$ as given by Claim \ref{claim:u'}. 
		Put $V := N_F(u) \setminus (N_F(u') \cup \{u'\})$ and $V' := N_F(u') \setminus (N_F(u) \cup \{u\})$; without loss of generality, $|V| \geq |V'|$ (else swap $u,u'$). So 
        $|V| \geq \frac{1}{2}|N_F(u) \triangle N_F(u')| - 1 \geq \frac{\varepsilon}{5}n$.    

        Next, we need to find two suitable vertices $x,x' \in V(G)$ onto which we will map $u,u'$. This is done in the following claim.

        \begin{claim}\label{claim:x,x'}
            There exist $x,x' \in V(K_n)$ with  
            $|N_G(x) \setminus N_G(x')|,
            |N_G(x') \setminus N_G(x)| < 0.99n$
            and $|N_G(x) \setminus \nolinebreak N_G(x')| > 0.01n$.
        \end{claim}
        \begin{proof}
            Let $x$ be a vertex of maximum degree in $G$. 
            Note that $d(x) \geq 0.49(n-1)$ by the assumption that $e(G)/\binom{n}{2} \geq 0.49$.
            Suppose first that $d(x) \leq 0.97n$. If there is $x' \in V(G)$ with 
            $|N_G(x) \setminus N_G(x')| > 0.01n$ then we are done; indeed, the other requirement in the claim holds because $\Delta(G) \leq 0.97n$. So suppose by contradiction that 
            $|N_G(x) \setminus N_G(x')| \leq 0.01n$ for all $x' \in V(G)$. Then every vertex $x'$ has at least $d(x) - 0.01n$ neighbors in $N_G(x)$. By averaging, there is a vertex $y \in N_G(x)$ with 
            $d(y) \geq \frac{n \cdot (d(x) - 0.01n)}{d(x)} > 0.97n$, where the last inequality holds for 
            $d(x) > n/3$.
            This contradicts the assumption that $\Delta(G) \leq 0.97n$. 

            Next, suppose that $d(x) \geq 0.97n$. 
            If there is $x' \in V(G)$ with 
            $0.01n < |N_G(x) \setminus N_G(x')| < 0.99n$ then we are done. 
            So suppose, for the sake of contradiction, that there is no such $x'$.
            Set
            $Z = \{x' : |N_G(x) \setminus N_G(x')| \geq 0.99n\}$ and 
            $W = \{x' : |N_G(x) \setminus N_G(x')| \leq 0.01n\}$. By assumption, $Z \cup W = V(G)$. Also, each vertex in $Z$ has degree at most $0.01n$, and each vertex in $W$ has degree at least $d(x) - 0.01n \geq 0.96n$.
            So $|W| \cdot 0.96n \leq 2e(G) \leq 1.02\binom{n}{2} \leq 0.51n^2$, implying that $|W| \leq 0.55n$. 
            Similarly, $0.49\binom{n}{2} \leq e(G) \leq |W| \cdot n + |Z| \cdot 0.01n \leq |W| \cdot n + 0.01n^2$, so $|W| \geq 0.2n$.
            Each $y \in W$ has at least $0.96n - |W| \geq 0.4n$ neighbors in $Z$. Hence, by averaging, there is $z \in Z$ with $d(z) \geq \frac{|W| \cdot 0.4n}{|Z|} \geq \frac{0.2n \cdot 0.4n}{n} > 0.01n$, a contradiction.
        \end{proof}
        
        Fix vertices $x,x' \in V(G)$ as given by Claim \ref{claim:x,x'}.
        Put $X := N_G(x) \setminus (N_G(x') \cup \{x'\})$ and $X' := N_G(x') \setminus (N_G(x) \cup \{x\})$, and assume without loss of generality that $|X| \geq |X'|$ (else swap $x$ and $x'$). 
        We have $|X| \geq 0.01n$ and $|X|,|X'| \leq 0.99n$.
		Consider disjoint $N,N' \subseteq V(K_n) \setminus \{x,x'\}$ with $|N| = |V|, |N'| = |V'|$, to be chosen later. Let $\varphi_1$ be an embedding of $F$ into $K_n$ where $\varphi_1(u) = x, \varphi_1(u') = x', \varphi_1(V) = N, \varphi_1(V') = N'$. Let $\varphi_2$ be obtained from $\varphi_1$ by switching the roles of $u,u'$ (while leaving all other vertices unchanged); i.e., 
		$\varphi_2(u) = x', \varphi_2(u') = x$ and $\varphi_2(v) = \varphi_1(v)$ for every $v \in V(F) \setminus \{u,u'\}$.
		For $i=1,2$, let $F_i$ be the copy of $F$ given by $\varphi_i$, i.e., $F_i = \{\varphi_i(e) : e \in E(F)\}$. 
		Observe that
		\begin{align}
		|E(F_1) \cap E(G)| &- |E(F_2) \cap E(G)| \nonumber \\ &= 
		|N \cap N_G(x)| + |N' \cap N_G(x')| - |N \cap N_G(x')| - |N' \cap N_G(x)|
		\label{eq:switching single pair 1}
		\\ &=
		|N \cap X| + |N' \cap X'| - |N \cap X'| - |N' \cap X|.
		\label{eq:switching single pair 2}
		\end{align} 
		Let us explain the two equalities above.
		To see that \eqref{eq:switching single pair 1} holds, consider the contribution of any $e \in E(F)$ to 
		$S := |E(F_1) \cap E(G)| - |E(F_2) \cap E(G)|$. If $e$ does not touch $u,u'$, or if $e = uu'$ (assuming $uu'$ is an edge), then $\varphi_1(e) = \varphi_2(e)$, so $e$ does not contribute to $S$. If $e$ connects $u$ or $u'$ to a vertex $v$ of $N_F(u) \cap N_F(u')$, then $\{\varphi_1(uv), \varphi_1(u'v)\} = \{\varphi_2(uv), \varphi_2(u'v)\}$, so again these edges do not contribute to $S$. So it remains to consider edges between $u$ and $N_F(u) \setminus (N_F(u') \cup \{u'\}) = V$ and between $u'$ and 
		$N_F(u') \setminus (N_F(u) \cup \{u\}) = V'$. The contribution of these edges to $S$ is precisely \eqref{eq:switching single pair 1}. 
        To see that \eqref{eq:switching single pair 2} holds, note that in the expression 
		\eqref{eq:switching single pair 1}, the contributions of 
		$|N \cap N_G(x) \cap N_G(x')|$ and 
        $|N' \cap N_G(x) \cap N_G(x')|$ cancel, leaving only the contributions coming from 
        $N_G(x) \setminus (N_G(x') \cup \{x'\}) = X$ and $N_G(x') \setminus (N_G(x) \cup \{x\}) = X'$.
		
		It remains to choose disjoint $N,N'$ (with $|N| = |V|, |N'| = |V'|$) for which \eqref{eq:switching single pair 2} is 
        $\Omega(\varepsilon n) \geq c \varepsilon n$. We consider two cases. Suppose first that $|X| \geq |N|$. In this case choose $N \subseteq X$, and choose $N'$ such that
        $|N' \cap X| \leq \max(0, |N'| - 0.01n)$. In other words, we place at least $\min(|N'|,0.01n)$ elements into $N'$ which do not belong to $X$;
        this is possible as 
        $|X| \leq 0.99n$. For this choice of $N,N'$, the expression \eqref{eq:switching single pair 2} is at least 
        $|N| - |N' \cap X| \geq 
        |N| - \max(0,|N'| - 0.01n) 
		\geq 
        \min(|N|,0.01n) \geq \frac{\varepsilon}{5}n$, using that 
        $|N| \geq |N'|$ and $|N| = |V| \geq \frac{\varepsilon}{5}n$. 
        Now suppose that $|N| \geq |X|$. Choose $N$ such that $N$ contains $X$ and also contains as few elements of $X'$ as possible. 
		As $|N| = |V| \leq d_F(u) \leq \Delta(F) \leq (1-\varepsilon)n$, we can choose $N$ such that $|N \cap X'| \leq \max(0, |X'| - \varepsilon n)$. Hence, using $X \subseteq N$, the expression \eqref{eq:switching single pair 2} is at least $|X| - |N \cap X'| \geq |X| - \max(0, |X'| - \varepsilon n) \geq \min(|X|,|X|-|X'|+\varepsilon n)$. As $|X| \geq |X'|$ and $|X| \geq 0.01n$, this is at least $\varepsilon n$, as required. This completes the proof in Case 1.

\paragraph{Case 2:} 
$\Delta(F) \leq \delta n$. 
In this case we will need to use that $G$ is host-$\beta$-good. Thus, we first handle the case that it is not. 
Indeed, suppose that there is a partition $V(G) = X \cup Y$ with $|X| = \lfloor \frac{n}{2} \rfloor$ satisfying 
$\big| e_G(X,Y) - \frac{1}{2}|X||Y| \big| \geq \gamma n^2$.
By Theorem \ref{thm:bisection max degree}, there is a partition $V(F) = U \cup V$ with $|U| = \lfloor \frac{n}{2} \rfloor$ and $|e_F(U,V) - \frac{e(F)}{2}| \geq \frac{n}{6}$.
By Lemma \ref{lem:cut discrepancy} with $t = \frac{n}{6}$, there is a copy $F_0$ of $F$ in $K_n$ with $\big|  |E(F_0) \cap E(G)| - \frac{e(F)}{2} \big| \geq 0.5\gamma t$. Thus, $F_0$ has discrepancy at least $\gamma t \geq c n$. Note that we may apply Lemma \ref{lem:cut discrepancy} because 
$\gamma t = \frac{\gamma n}{6}$ while
$e(F)/n \leq \Delta(F) \leq \delta n$, and $\gamma \ll \delta$.  
Thus, from now on, we may assume that $\big| e_G(X,Y) - \frac{1}{2}|X||Y| \big| \leq \gamma n^2$ for every partition $V(G) = X \cup Y$ with $|X| = \lfloor \frac{n}{2} \rfloor$. By Lemma \ref{lem:host-good}, $G$ is host-$\beta$-good. 

Next, we consider two subcases, based on the independence number of $F$. 
\paragraph{Case 2.1:}
$\alpha(F) \geq \frac{n}{2}$. Let $U \subseteq V(F)$ be an independent set of size $\lfloor \frac{n}{2} \rfloor$ and let $V = V(F) \setminus U$. 
By assumption, $F$ has no isolated vertices, so $d_V(u) \geq 1$ for every $u \in U$.
Let $u_i,u'_i \in U$, $i = 1,\dots,m$, be a maximal collection of pairs of vertices such that $u_1,u'_1,\dots,u_m,u'_m$ are pairwise distinct and satisfy
$N_F(u_i) \setminus N_F(u'_i) \neq \emptyset$ for all $i$. We claim that $m \geq 0.05n$. Suppose not. Fix any 
$u \in U_0 := U \setminus \{u_i,u'_i : 1 \leq i \leq m\}$. Let $v \in V$ be a neighbor of $u$. By assumption, $v$ is adjacent to all $u' \in U_0$, so $d(v) \geq |U_0| = |U| - 2m \geq 0.4n$, contradicting the assumption of Case 2. Now, the partition $(U,V)$ and the vertices $u_i,u'_i$, $i = 1,\dots,m := 0.05n$, witness the fact that $F$ is guest-good. Indeed, $U$ is independent, and setting $d_i := 1$ for all $i$, we have that 
$|N_V(u_i) \setminus N_V(u'_i)| \geq 1 = d_i$ and 
$d_V(u_i),d_V(u'_i) \leq \Delta(F) \leq \frac{2}{3}|V|$ for all $i$. The value of $F$ (as in Definition \ref{def:guest}) is $t := \sum_{i=1}^m \sqrt{d_i} = m = 0.05n$. By Lemma \ref{lem:main}, there is a copy $F_0$ of $F$ in $K_n$ such that 
$\big|  |E(F_0) \cap E(G)| - \frac{e(F)}{2} \big| \geq \beta t$. Hence, $F_0$ has discrepancy at least 
$2 \beta t \geq c n$, completing the proof in Case 2.1.

\paragraph{Case 2.2:}
$\alpha(F) \leq \frac{n}{2}$. 
Let $A$ be a maximal independent set in $F$; so $|A| \leq \frac{n}{2}$. By the maximality of $A$, every vertex outside $A$ has a neighbor in $A$. 
In Case 2.2, we will switch a subset of the pairs of vertices $u_iu'_i$, $i=1,\dots,m$, supplied by the following claim.
The purpose of Item 1 in Claim \ref{claim:switching vertices} is to guarantee that the switching results in a large gain (i.e., a large difference in the number of edges of $G$ in the copy of $F$ under consideration). The purpose of Item 2 is that at the same time, the total number of vertices participating in the switching is small. These vertices will have to be chosen carefully, and having only few of them is helpful in obtaining such a choice. 

\begin{claim}\label{claim:switching vertices}
   There exist distinct vertices $u_i,u'_i \in A$, $i=1,\dots,m \leq 5\delta n$, such that the following holds. For $1 \leq i \leq m$, let $W_i$ be the set of $v \in V(F)$ which are adjacent to $u_i$ or $u'_i$, but not adjacent to $u_j$ or $u'_j$ for any $j < i$. Let also $V_i = W_i \setminus N_F(u'_i)$ and $V'_i = W_i \setminus N_F(u_i)$.\footnote{Thus, $V_i$ is the set of vertices of $W_i$ which are adjacent to $u_i$ but not $u'_i$, and similarly for $V'_i$.}
   Then 
   \begin{enumerate}
        \item $\sum_{i=1}^m |V_i| \geq \delta n$.
        \item $\sum_{i=1}^m |W_i| \leq 10\delta n$.
   \end{enumerate}
\end{claim}
Note that the sets $W_i$ in the statement of the claim are pairwise-disjoint, and that $\bigcup_{i=1}^m W_i = \bigcup_{u \in U}N_F(u)$, where $U := \{u_i,u'_i : i=1,\dots,m\}$. Also, $W_i \cap A = \emptyset$ because $A$ is independent.
\begin{proof}[Proof of Claim \ref{claim:switching vertices}]
    We have $e(A,V(F) \setminus A) \leq |A| \cdot \Delta(F) \leq |A| \cdot \delta n$, and therefore, the number of $v \in V(F) \setminus A$ with $d_A(v) \geq \frac{|A|}{5}$ is at most 
    $5\delta n$. Let $B$ be the set of all $v \in V(F) \setminus A$ with $d_A(v) \leq \frac{|A|}{5}$, so 
    $|B| \geq n - |A| - 5\delta n \geq (0.5-5\delta)n$. 

    Let $U = \{u_i,u'_i : i=1,\dots,m\}$ be a maximal collection of distinct vertices of $A$ satisfying that $3|V_i \cap B| \geq |W_i \cap B| \geq \frac{n}{5|A|}$ for every $i \in [m]$, and that $|W| \leq 5\delta n$, where 
    $W := \bigcup_{i=1}^m (W_i \cap B)$ and $W_i$ is defined as in the statement of the claim. 
    We have $\sum_{i=1}^m |W_i| \leq |W| + 
    |V(F) \setminus (A \cup B)| \leq 10\delta n$. As $W_i \cap B$ are pairwise disjoint and non-empty, it also follows that 
    $m \leq |W| \leq 5\delta n$. 
    Hence, if $\sum_{i=1}^m |V_i| \geq \delta n$ then the statement of the claim holds. 
    So suppose otherwise. 
    Then $|W| = \sum_{i=1}^m |W_i \cap B| \leq 3\sum_{i=1}^m|V_i \cap B| \leq 3\delta n$.
    Note also that $W = \bigcup_{u \in U}(N_F(u) \cap B)$.
    Put $A' := A \setminus U$ and 
    $B' := B \setminus W$. 
    We have $|B'| = |B| - |W| \geq |B| - 3\delta n \geq (1-7\delta)|B|$.
    Since $V_i \cap B$ are pairwise disjoint and 
    $|V_i \cap B| \geq \frac{n}{15|A|}$ for every $i \in [m]$, we have 
    $m \leq \frac{15|A|}{n} \cdot \sum_{i=1}^m |V_i| \leq 15\delta |A|$.
    Hence, $|A'| = |A| - 2m \geq (1-30\delta)|A|$. 
    Moreover, every vertex $v \in B'$ has a neighbor in $A$, and this neighbor cannot be in $U$ (as $v \notin W$), so $v$ has a neighbor in $A'$. Thus, $e(A',B') \geq |B'|$.
    Now consider the expression
    $$
    S := \sum_{u,u' \in A'} \left( |N_{B'}(u)| - 2|N_{B'}(u) \cap N_{B'}(u')| - 
    \frac{0.5|B'|}{|A'|}\right),
    $$
    where the sum is over all ordered pairs of distinct vertices $u,u'$.
    We claim that $S \geq 0$. 
    Indeed, let $P$ denote the number of (ordered) triples $(u,u',v)$ with $u,u' \in A'$, $v \in B'$ and $v \sim u,u'$. Then
    $$
    S = e(A',B') \cdot (|A'|-1) - 2P - 0.5|B'|(|A'|-1).
    $$
    Also, $P \leq e(A',B') \cdot \frac{|A|}{5} \leq e(A',B') \cdot0.21|A'|$, because every $v \in B' \subseteq B$ has at most $\frac{|A|}{5}$ neighbors in $A$. It follows that 
    $S \geq e(A',B') \cdot (|A'|-1) - 0.42 \cdot e(A',B') \cdot |A'| - 0.5|B'|(|A'|-1) \geq 
    0,$
    using $e(A',B') \geq |B'|$.
    Now, as $S \geq 0$, there is a choice of $u,u' \in A'$ for which 
    $|N_{B'}(u)| \geq \frac{0.5|B'|}{|A'|} \geq 
    \frac{n}{5|A|}$ and 
    $|N_{B'}(u)| \geq 2|N_{B'}(u) \cap N_{B'}(u')|$. 
    Without loss of generality, assume that $|N_{B'}(u)| \geq |N_{B'}(u')|$ (else swap $u,u'$).
    Set $u_{m+1} := u$ and $u'_{m+1} := u'$. Note that $W_{m+1} \cap B = N_{B'}(u) \cup N_{B'}(u')$ (because every vertex in $B \setminus B' = W$ is adjacent to $u_j$ or $u'_j$ for some $j \leq m$). Similarly, $V_{m+1} \cap B = N_{B'}(u) \setminus N_{B'}(u')$. 
    It follows that  
    $|W_{m+1} \cap B| \geq |N_{B'}(u)| \geq \frac{n}{5|A|}$ and
    $|W_{m+1} \cap B| = |N_{B'}(u)| + |N_{B'}(u')| - |N_{B'}(u) \cap N_{B'}(u')| \leq 
    2|N_{B'}(u)| - |N_{B'}(u) \cap N_{B'}(u')| \leq 3|N_{B'}(u) \setminus N_{B'}(u')| = 3|V_{m+1} \cap B|$, where the last inequality uses that
    $|N_{B'}(u)| \geq 2|N_{B'}(u) \cap N_{B'}(u')|$. 
    Finally 
    $|\bigcup_{i=1}^{m+1}(W_i \cap B)| = |W| + |W_{m+1}| \leq |W| + 2\Delta(F) \leq 5\delta n$. Hence, we get a contradiction to the maximality of $m$. This proves the claim. 
\end{proof}

Let $U = \{u_i,u'_i : i = 1,\dots,m\}$ be the set of vertices given by Claim \ref{claim:switching vertices}, and let $W_i,V_i,V'_i$ be as in that claim. 
We will obtain a copy of $F$ with high discrepancy by switching the vertices $u_i,u'_i$ for certain indices $i \in [m]$.\footnote{The main difference between the argument we use here and the one used in the proof of Lemma \ref{lem:main}, is that here we will embed the neighborhoods of $u_i,u'_i$ in a certain greedy fashion, rather than randomly. This is done because $m$ may be small, and hence a random embedding may not give us linear discrepancy.} As $G$ is host-$\beta$-good, there are distinct vertices $x_i,x'_i \in V(G)$, $i = 1,\dots,m$, such that 
$0.01 n \leq |N_G(x_i) \triangle N_G(x'_i)| \leq 0.99n$ for every $i \in [m]$.
Without loss of generality, 
$|N_G(x_i) \setminus N_G(x'_i)| \geq \frac{n}{200}$ for every $i \in [m]$.
Put $X := \{x_i,x'_i : i \in [m]\}$. 
In the following claim, we show that $\bigcup_{u \in U}N_F(u)$ can be embedded into $V(K_n)$ so that switching each pair $u_i,u'_i$ will result in a significant change in the number of edges of $G \subseteq K_n$.
\begin{claim}\label{claim:switching greedy embedding}
    There is an injection 
    $g : \bigcup_{u \in U}N_F(u) \rightarrow V(G) \setminus X$ such that for all $i \in [m]$, the quantity
    \begin{equation*}
        D_i := |g(N_F(u_i)) \cap N_G(x_i)| + 
        |g(N_F(u'_i)) \cap N_G(x'_i)| - 
        |g(N_F(u_i)) \cap N_G(x'_i)| - 
        |g(N_F(u'_i)) \cap N_G(x_i)|
    \end{equation*}
    satisfies $|D_i| \geq 0.5|V_i|$.
\end{claim}
\begin{proof}
    Recall that $\bigcup_{u \in U}N_F(u) = \bigcup_{i=1}^m W_i$. We will define the map $g$ step by step; at the $i$th step we will define $g(W_i)$, and at this step we will also make sure that $|D_i| \geq 0.5|V_i|$. So let $1 \leq i \leq m$, and suppose that we already defined $g(W_j)$ for $1 \leq j < i$.
    Let $W = \bigcup_{j < i} W_j$, and note that
    $|W| \leq 10\delta n$ by the second item of Claim \ref{claim:switching vertices}. 
    Set also $Z_i := W \cap N_F(u_i)$ and $Z'_i = W \cap N_F(u'_i)$. Then, setting
    $$
    S := 
    |g(Z_i) \cap N_G(x_i)| 
    + 
    |g(Z'_i) \cap N_G(x'_i)|
    - 
    |g(Z_i) \cap N_G(x'_i)| 
    - 
    |g(Z'_i) \cap N_G(x_i)|
    $$
    and 
    $$
    T := 
    |g(V_i) \cap N_G(x_i)| + 
    |g(V'_i) \cap N_G(x'_i)| - 
    |g(V_i) \cap N_G(x'_i)| - 
    |g(V'_i) \cap N_G(x_i)|,
    $$
    we have
    $
    D_i = S + T.
    $
    Here we used that $W_i \setminus (V_i \cup V'_i) \subseteq N_F(u_i) \cap N_F(u'_i)$, and that the vertices in $N_F(u_i) \cap N_F(u'_i)$ do not contribute to $D_i$, as their contribution cancels out. 
    Note that $S$ is determined by the choice of $g(W_j)$ for $j < i$.
    Now there are two cases: If $S \leq -0.5|V_i|$, then choose $g(V_i),g(V'_i)$ so that $g(V_i),g(V'_i)$ are contained in the complement of $N_G(x_i) \triangle N_G(x'_i)$. Such a choice of $g(V_i),g(V'_i)$ exists with $g(V_i),g(V'_i)$ disjoint to all previously embedded vertices and to $X$. Indeed, the total number of embedded vertices is at most $\sum_{j=1}^m|W_j| \leq 10\delta n$, and also $|X| = 2m \leq 10\delta n$ (here we use Claim \ref{claim:switching vertices}); while on the other hand,
    $\big| \overline{N_G(x_i) \triangle N_G(x'_i)} \big| \geq 0.01n \geq 20\delta n$. 
    Now, $g(V_i),g(V'_i) \subseteq \overline{N_G(x_i) \triangle N_G(x'_i)}$ means that $T = 0$ (because the first term in the definition of $T$ cancels with the third, and the second with the fourth). Hence, $D_i = S \leq -0.5|V_i|$, as required. 

    Now suppose that $S \geq -0.5|V_i|$. Choose $g(V_i)$ to be contained in $N_G(x_i) \setminus N_G(x'_i)$, and choose $g(V'_i)$ to be contained in $\overline{N_G(x_i) \triangle N_G(x'_i)}$. Again, we can choose such $g(V_i),g(V'_i)$ which are disjoint from all previously embedded vertices and from $X$, as $|N_G(x_i) \setminus N_G(x'_i)|, \big| \overline{N_G(x_i) \triangle N_G(x'_i)} \big| \geq \frac{n}{200} \geq 20\delta n$. The choice of $g(V_i),g(V'_i)$ implies that 
    $T = |g(V_i) \cap N_G(x_i)| - |g(V_i) \cap N_G(x'_i)| = |V_i| - 0 = |V_i|$. Hence, $D_i = S+T \geq 0.5|V_i|$, as required.
\end{proof}

We now complete the proof in Case 2.2, and thus the proof of the theorem. 
Let $g$ be the injection given by Claim \ref{claim:switching greedy embedding}.
Let $f_1 : V(F) \rightarrow V(G)$ be a bijection which satisfies $f_1(u_i) = x_i, f_1(u'_i) = x'_i$ for every $i \in [m]$ and agrees with $g$ on $\bigcup_{u \in U}N_F(u)$. (The images of the other vertices of $F$ are arbitrary.) Let $D_i$ be defined as in Claim \ref{claim:switching greedy embedding}. Let $I = \{i \in [m] : D_i \geq |V_i|/2\}$ and 
$J = [m] \setminus I = \{i \in [m] : D_i \leq -|V_i|/2\}$. Without loss of generality, 
$\sum_{i\in I}D_i \geq 0.25\sum_{i=1}^m |V_i|$; the complementary case, i.e. that 
$\sum_{i\in J}D_i \leq -0.25\sum_{i=1}^m |V_i|$, is handled symmetrically. Now, let $f_2$ be obtained from $f_1$ by switching $u_i,u'_i$ for every $i \in I$; namely, $f_2(u_i) = x'_i, f_2(u'_i) = x_i$ for all $i \in I$, and $f_2(v) = f_1(v)$ for all $v \in V(F) \setminus \{u_i,u'_i : i\in I\}$. 
For $i=1,2$, let $F_i = \{f_i(e) : e \in E(F)\}$ be the copy of $F$ corresponding to $f_i$. 
Then 
$$
|E(F_1) \cap E(G)| - |E(F_2) \cap E(G)| = \sum_{i \in I}D_i \geq 0.25\sum_{i=1}^m |V_i| \geq 0.25\delta n \geq 2c n.
$$
Here, the equality uses that $\{u_i,u'_i : i \in [m]\} \subseteq A$ is independent, and the second inequality uses Item 1 of Claim \ref{claim:switching vertices}. Hence, for some $i=1,2$, $F_i$ has discrepancy at least $cn$, as required. 
\end{proof}

\section{$K_k$- and $2$-factors}\label{sec:K_k factor}

    In this section we prove Theorems \ref{thm: Kk-factors} and \ref{prop:2-factor}. It will be convenient to refer to the two colors as red and blue, rather than $1,-1$. The following graphs play an important role in the proof of both theorems. 

    \begin{definition}
    For an integer $m$, let $\mathcal{F}_m$ denote the family of red/blue-colored complete graphs on $2m$ vertices such that one color is either a complete graph of size $m$ or the union of two disjoint complete graphs of size $m$ each.
    \end{definition}

    Note that $\mathcal{F}_m$ has exactly four nonisomorphic elements. Moreover, two of the elements (i.e., those where one of the colors forms a complete graph of size $m$) are the bipartite construction with ratio \nolinebreak $\frac{1}{2}$ (see Definition \ref{def:bipartite construction}).
    
    Our proof strategy for Theorems~\ref{thm: Kk-factors}-\ref{prop:2-factor} builds upon the following theorem by Cutler and Mont{\'a}gh \cite{Unavoidable} stating that as long as a coloring of $K_n$ is not very imbalanced, it contains a member of $\mathcal{F}_m$.
    \begin{theorem}[\cite{Unavoidable}]\label{thm: get small graph}
        For every $m \geq 1$ and $\varepsilon > 0$, there is $n_0=n_0(m,\eps)$ such that the following holds for all $n \geq n_0$. Every red/blue-coloring of $E(K_n)$ with at least $\eps \binom{n}{2}$ edges of each color contains an element of $\mathcal{F}_m$.
    \end{theorem}
    The above result has a simple proof: applying Ramsey's theorem repeatedly, we can partition all but $O(1)$ of the vertices of $K_n$ into monochromatic cliques of size $M$, where $M$ is a large constant compared to $m$. If there are both a red and a blue clique, we get a member of $\mathcal{F}_m$ by applying the K{\H{o}}v{\'a}ri-S{\'o}s-Tur{\'a}n theorem (\cite{kovari-sos-turan}) to the denser color in-between these two cliques. Otherwise, suppose that all cliques are red. By averaging, there exist two red cliques such that the blue bipartite graph in-between them has at least $\eps/2\cdot M^2$ edges. Therefore, the K{\H{o}}v{\'a}ri-S{\'o}s-Tur{\'a}n theorem again gives a member of $\mathcal{F}_m$.
    As demonstrated by Fox and Sudakov \cite{fox08}, a better bound on $n_0$ can be obtained using dependent random choice.
    
    We will apply Theorem~\ref{thm: get small graph} iteratively, removing members of $\mathcal{F}_m$ one-by-one. In the end, we will have a vertex-partition of $K_n$ with one part (called $W$) being potentially large but almost monochromatic, and the other parts (denoted as $V_1,\ldots,V_s$) being constant-sized and having one of four very specific colorings (namely, the graph induced by $V_i$ belongs to $\mathcal{F}_m$ for every $i$). To obtain a factor with high discrepancy, we will then need to understand the kinds of factors contained in the members of $\mathcal{F}_m$.
    
    
    \subsection{$K_k$-factors: Proof of Theorem \ref{thm: Kk-factors}}\label{subsec:K_k-factors}
	First we establish some facts related to the bipartite construction and $\lambda_k$. 
	\begin{lemma}\label{lem:K_k factor in bipartite construction}
		Let $n$ be divisible by $k$, and consider the $n$-vertex bipartite construction with ratio $\rho$ and with parts $X,Y$, where all edges touching $X$ are red and all edges inside $Y$ are blue. Then
		\begin{enumerate}
            \item Writing $|Y| = qk + r$ with $0 \leq r < k$, the largest number of blue edges in a $K_k$-factor is $q\binom{k}{2} + \binom{r}{2} = \frac{k-1}{2}(1-\rho)n + O(1)$.
			\item The largest number of red edges in a $K_k$-factor is 
			\begin{equation}\label{eq:red edges K_k factor bipartite construction}
			\left( \frac{k-1}{2} - \frac{(k-\lfloor k\rho \rfloor-1)}{2}\left( \frac{\lfloor k\rho \rfloor}{k} + 1 - 2\rho \right) \right)n,
			\end{equation}
		\end{enumerate}
	\end{lemma}
	\begin{proof}
		For the first item, consider any $K_k$-factor $\mathcal{F}$, and let $Y_1,\dots,Y_{n/k}$ be the intersections with $Y$ of the $k$-cliques in $\mathcal{F}$. Then $\sum_{i=1}^{n/k}|Y_i| = |Y|$, and the number of blue edges in $\mathcal{F}$ is 
        $\sum_{i=1}^{n/k} \binom{|Y_i|}{2}$. As the function $x \mapsto \binom{x}{2}$ is convex, the sum 
        $\sum_{i=1}^{n/k} \binom{|Y_i|}{2}$ under the constraints $\sum_{i=1}^{n/k}|Y_i| = |Y|$ and $0 \leq |Y_i| \leq k$ is maximized when at most one $|Y_i|$ is not equal to the maximum value, namely $k$.
        Thus, the number of blue edges is at most $q\binom{k}{2} + \binom{r}{2}$, and this value is attainable by taking a $K_k$-factor which decomposes $Y$ into $q$ cliques of size $k$ and one clique of size $r$. 
        
        For the second item, fix a $K_k$-factor $\mathcal{F}$, and for each $0 \leq i \leq k$, let $a_i$ be the number of $k$-cliques in $\mathcal{F}$ with exactly $i$ vertices in $X$. Then $\sum_{i=0}^k a_i = \frac{n}{k}$,
		$\sum_{i=1}^k i \cdot a_i = |X| = \rho n$, and 
		$e_{\text{red}}(\mathcal{F}) = \sum_{i=1}^{k} (\binom{k}{2} - \binom{k-i}{2})a_i$. Suppose that $\mathcal{F}$ maximizes the number of red edges. We claim that then there are no $0 \leq i,j \leq k$ with $j \geq i+2$ and $a_i,a_j > 0$. Suppose otherwise. Let $\mathcal{F'}$ be the $K_k$-factor obtained from $\mathcal{F}$ by removing one $k$-clique intersecting $X$ in $i$ vertices and one $k$-clique intersecting $X$ in $j$ vertices, and instead adding two $k$-cliques intersecting $X$ in $i+1$ and $j-1$ vertices, respectively. Then
		\begin{align*}
        e_{\text{red}}(\mathcal{F}') - e_{\text{red}}(\mathcal{F}) &= -\binom{k-i-1}{2} - \binom{k-j+1}{2} + \binom{k-i}{2} + \binom{k-j}{2} \\ &= (k-i-1) - (k-j) = j-i-1 > 0, 
		\end{align*}
		contradicting the maximality of $\mathcal{F}$. 
        Now, letting $0 \leq j \leq k$ be minimal with $a_j > 0$, we get that $a_i = 0$ for all $i \notin \{j,j+1\}$. So $a_j + a_{j+1} = \frac{n}{k}$ and $j a_j + (j+1)a_{j+1} = \rho n$. Solving this system of equations, we get $a_j = (\frac{j+1}{k}-\rho)n$ and 
		$a_{j+1} = (\rho - \frac{j}{k})n$.
        Since $a_j > 0$ and $a_{j+1} \geq 0$, it follows that 
		$j = \lfloor k\rho \rfloor$. Also, we have
		\begin{equation}\label{eq:red edges K_k-factor bipartite construction aux}
		\begin{split} 
		e_{\text{red}}(\mathcal{F}) &= 
		\left( \binom{k}{2} - \binom{k-j}{2} \right)a_j + 
		\left(\binom{k}{2} - \binom{k-j-1}{2} \right)a_{j+1} \\ &= \binom{k}{2}\frac{n}{k} - \binom{k-j}{2} a_j - \binom{k-j-1}{2}a_{j+1}.
		\end{split}
		\end{equation}
        We calculate:
		\begin{align*}
		\binom{k-j}{2} a_j + \binom{k-j-1}{2}a_{j+1} &= 
		\binom{k-j}{2}\left(\frac{j+1}{k}-\rho \right)n + \binom{k-j-1}{2}\left(\rho - \frac{j}{k}\right)n \\ &= 
		\frac{(k-j-1)n}{2} \cdot \left( (k-j)\left(\frac{j+1}{k}-\rho\right) + (k-j-2)\left(\rho - \frac{j}{k}\right) \right) \\ &=
		\frac{(k-j-1)n}{2}\left( \frac{j}{k} + 1 - 2\rho \right).
		\end{align*}
		By combining this with \eqref{eq:red edges K_k-factor bipartite construction aux} and plugging in $j = \lfloor k\rho \rfloor$, we get
		$$
		e_{\text{red}}(\mathcal{F}) = \left( \frac{k-1}{2} - \frac{(k-\lfloor k\rho \rfloor-1)}{2}\left( \frac{\lfloor k\rho \rfloor}{k} + 1 - 2\rho \right) \right)n,
		$$
  as required.
	\end{proof}
	\noindent
	Using Lemma \ref{lem:K_k factor in bipartite construction}, we can determine $\lambda_k$. We also prove an upper bound on $\lambda_k$ that will be used in the proof of Theorem \ref{thm: Kk-factors}. 
	\begin{lemma}\label{lem:lambda_k}
		For $k \geq 2$, $\lambda_k$ equals $\frac{k-1}{2}(1-\rho_k)$, where $\rho_k$ is the unique solution in $[0,1]$ to the \nolinebreak equation
		\begin{equation}\label{eq:rho equation}
		(k-1)\rho = (k-\lfloor k\rho \rfloor-1)\left( \frac{\lfloor k\rho \rfloor}{k} + 1 - 2\rho \right).
		\end{equation}
		Moreover, $\lambda_k \leq \frac{k-1}{3}$. 
	\end{lemma}
	\begin{proof}
		Let $f(\rho) := \frac{k-1}{2}(1-\rho)$ and $g(\rho) := \frac{k-1}{2} - \frac{k-\lfloor k\rho \rfloor-1}{2}\left( \frac{\lfloor k\rho \rfloor}{k} + 1 - 2\rho \right)$ (as in \eqref{eq:red edges K_k factor bipartite construction}). By Lemma \ref{lem:K_k factor in bipartite construction} and the definition of $\lambda_k$, we have $\lambda_k = \min_{\rho \in [0,1]} \max (f(\rho),g(\rho))$. The function $f$ is decreasing, while the function $g$ is increasing (this can be seen from the fact that $g(\rho)n$ is the maximum number of red edges in a $K_k$-factor of the $n$-vertex bipartite construction with ratio $\rho$, and the set of red edges of this construction grows when increasing $\rho$). It follows that the minimum of $\max (f(\rho),g(\rho))$ is obtained at the unique point $\rho$ where $f(\rho) = g(\rho)$. Rearranging this equation gives \eqref{eq:rho equation}. 
		
		We now prove that $\rho_k \geq \frac{1}{3}$, which would imply that $\lambda_k = \frac{k-1}{2}(1-\rho_k) \leq \frac{k-1}{3}$, as required. 
		To prove that $\rho_k \geq \frac{1}{3}$, it suffices to show that $f(\frac{1}{3}) \geq g(\frac{1}{3})$. 
        Set $a := \lfloor \frac{k}{3} \rfloor$. Then
        $$
        g\left(\frac{1}{3} \right) = \frac{k-1}{2} - \frac{k-a-1}{2}\left(\frac{a}{k} + \frac{1}{3}\right) = 
        \frac{k-1}{3} + \frac{a}{6k} \cdot (3a - 2k + 3).
        $$
        Note that $3a-2k+3 \leq 0$. Indeed, for $k \geq 3$ we have $3a-2k+3 \leq -k+3 \leq 0$, and for $k=2$ we have $a=0$. 
        Hence, $g\left(\frac{1}{3} \right) \leq \frac{k-1}{3} = f\left(\frac{1}{3} \right)$.
	\end{proof}
    \noindent
	See Figure \ref{fig:lambda_k} for the values of $\lambda_k$ and $\rho_k$ for small $k$. 
	\begin{figure}[h]
	\begin{center}
		\begin{tabular}{ c | c | c | c | c | c }
			$k$ & $2$ & $3$ & $4$ & $5$ & $6$ \\ \hline
			$\rho_k$ & $\frac{1}{3}$ & $\frac{1}{3}$ & $\frac{5}{14}$ & $\frac{9}{25}$ & $\frac{4}{11}$ \\ \hline
			$\lambda_k$ & $\frac{1}{3}$ & $\frac{2}{3}$ & $\frac{27}{28}$ & $\frac{32}{25}$ & $\frac{35}{22}$ 
		\end{tabular}
		\caption{First few values of $\rho_k$ and $\lambda_k$.}\label{fig:lambda_k}
	\end{center}
	\end{figure}
    

    Next, we observe that $\rho_k$ is rational and use this to show that the supremum in the definition of $\lambda_k$ is in fact a maximum.
    
	\begin{lemma}\label{lem: existence of bip construction}
        There exists a bipartite construction $B$ on $m$ vertices which contains a $K_k$-factor with at least $\lambda_k m$ red edges as well as a $K_k$-factor with at least $\lambda_k m$ blue edges. 
	\end{lemma}
    \begin{proof}
        Note that \eqref{eq:rho equation} is linear in $\rho$ for $\rho\in[(i-1)/k,i/k)$ for $i\in[k]$. Therefore, $\rho_k$ is a rational number with denominator say $M$. Set $m=kM$ and let $B$ be the bipartite construction on $kM$ vertices with ratio $\rho_k$ and parts $X,Y$, where all the edges touching $X$ are red and all edges inside $Y$ are blue. Note that this exists since $\rho_k M$ is an integer. Furthermore, we have that $|Y| = (1-\rho_k)kM$ is divisible by $k$. By Item 1 of Lemma~\ref{lem:K_k factor in bipartite construction}, we get that $B$ contains a $K_k$-factor with $\frac{k-1}{2}(1-\rho_k) m$ blue edges. By Lemma~\ref{lem:lambda_k}, this is equal to $\lambda_k m$. Finally, by \eqref{eq:rho equation} and Item 2 of Lemma~\ref{lem:K_k factor in bipartite construction}, we get that $B$ contains a $K_k$-factor with $\lambda_k m$ red edges.
    \end{proof}

    It is easy to see that any $m$-vertex bipartite construction (with any ratio $\rho$) is contained in a $2m$-vertex bipartite construction with ratio $\frac{1}{2}$, which is an element of $\mathcal{F}_m$. 
    This makes these elements of $\mathcal{F}_m$ useful for our purpose. 
    Let us now consider the other two elements of $\mathcal{F}_m$, which consist of two disjoint $m$-cliques in one color connected by a complete bipartite graph in the other color. 
    We consider the $K_k$-factors in this construction.
    \begin{lemma}\label{lem: Kk factor in two cliques}
        Let $C$ be the red/blue-colored complete graph on $2m$ vertices with $m$ divisible by $k$ such that the red edges induce two disjoint complete graphs of size $m$ each. Then $C$ contains a $K_k$-factor with $\frac{2m}{k}\cdot \binom{k}{2}$ red edges as well as a $K_k$-factor with $\frac{2m}{k}\cdot \lceil \frac{k}{2} \rceil\lfloor \frac{k}{2} \rfloor$ blue edges.
    \end{lemma}
    \begin{proof}
        Let $X,Y$ denote the vertex-set of the two complete graphs induced by the red edges. The former $K_k$-factor is found by having each copy of $K_k$ be contained in either $X$ or $Y$, the latter by having each copy of $K_k$ intersect one of $X,Y$ in $\lceil \frac{k}{2} \rceil$ and the other in $\lfloor \frac{k}{2} \rfloor$ vertices.
    \end{proof}
    \noindent
    We are now ready to prove Theorem \ref{thm: Kk-factors}. 
    \begin{proof}[Proof of Theorem~\ref{thm: Kk-factors}]
        Fix any red/blue edge-coloring of $K_n$. Let $\varepsilon > 0$ be arbitrary, and let us show that if $n$ is large enough, then there is a $K_k$-factor with at least $(\lambda_k - \varepsilon)n$ edges of the \nolinebreak same \nolinebreak color. 
        
        Let $B$ be the bipartite construction given by Lemma~\ref{lem: existence of bip construction}, $m$ be the number of vertices in $B$, and $\overline{B}$ be the coloring of $K_m$ obtained by switching the colors in $B$. 
        Let $D$ be the red/blue-colored $K_{2m}$ where the blue edges form a clique of size $m$, and let $\overline{D}$ be the coloring obtained from $D$ by switching the colors (so $D,\overline{D}$ are bipartite constructions with ratio $\frac{1}{2}$). Observe that $B$ is contained in $D$ and $\overline{B}$ is contained in $\overline{D}$.
        Finally, let $C$ be as in Lemma~\ref{lem: Kk factor in two cliques} on $2m$ vertices and $\overline{C}$ obtained by switching the colors in $C$. 
        Then $\mathcal{F}_m = \{D,\overline{D},C,\overline{C}\}$.
        Hence, by Theorem~\ref{thm: get small graph}, either some color has at most $\eps \binom{n}{2}$ edges, or the coloring contains an element of $\mathcal{F}_{2m}$, and thus contains a copy $H$ of $B$, $\overline{B}$, $C$ or $\overline{C}$. 
        We write $V_1=V(H)$ and remove $V_1$ from the complete graph. We repeat this for as long as possible, eventually getting disjoint sets of vertices $V_1,\ldots,V_s$ and a set $W$ of remaining vertices such that each $V_i$ induces a copy of $B$, $\overline{B}$, $C$ or $\overline{C}$. Additionally, we have that either $|W|\leq n_0$ or some color $c$ has at most $\eps \binom{|W|}{2}$ edges on $W$. If we are in the latter case, without loss of generality, suppose that $c$ is blue. Let $I_{red}\subseteq [s]$ be the set of indices $i$ such that $V_i$ induces a copy of $C$, define $I_{blue}$ similarly with $\overline{C}$, and let $J=[s]\setminus (I_{red}\cup I_{blue})$. Let $M = 2m|I_{red}|+2m|I_{blue}|+|W|$ and set
        $$
        \alpha := \frac{2m|I_{blue}|}{M}.
        $$ 
        Note that $m|J| = n-M$. Based on the value of $\alpha$, we now construct two different $K_k$-factors. 
        \paragraph{Case 1:} $\alpha > 2/3$. We construct a $K_k$-factor with many blue edges. By Lemma~\ref{lem: existence of bip construction}, for each $i\in J$, there exists a $K_k$-factor $F_i$ on $V_i$ with at least $\lambda_k m$ blue edges. By Lemma~\ref{lem: Kk factor in two cliques}, for $i\in I_{blue}$, there exists a $K_k$-factor $F_i$ with $\frac{2m}{k}\cdot \binom{k}{2}$ blue edges. Let $F$ be any extension of the union of the above partial $K_k$-factors to a $K_k$-factor. We get that the number of blue edges in $F$ is at least 
        \begin{align*}
            |J|\cdot \lambda_km + |I_{blue}|\cdot \frac{2m}{k} \cdot\binom{k}{2}
            =(n-M)\cdot \lambda_k +\alpha M \cdot \frac{k-1}{2}
            >(n-M)\cdot \lambda_k +M\cdot \frac{k-1}{3} \geq 
            \lambda_k n,
        \end{align*}
        where the last inequality uses that $\lambda_k \leq \frac{k-1}{3}$ by Lemma~\ref{lem:lambda_k}. This completes the proof in Case 1.
        \paragraph{Case 2:} $\alpha \leq 2/3$. We construct a $K_k$-factor with many red edges. By Lemma~\ref{lem: existence of bip construction}, for each $i\in J$, there exists a $K_k$-factor $F_i$ on $V_i$ with at least $\lambda_k m$ red edges. Using Lemma~\ref{lem: existence of bip construction} twice, we get that for $i\in I_{red}$, there exists a $K_k$-factor $F_i$ on $V_i$ with $2m/k\cdot \binom{k}{2}$ red edges and, for $i\in I_{blue}$, there exists a $K_k$-factor $F_i$ on $V_i$ with $2m/k\cdot \lceil k/2\rceil\lfloor k/2\rfloor$ red edges. Finally, by Lemma~\ref{lem: random embedding}, there exists a $K_k$-factor $F_W$ on $W$ with at least $(1-\eps)|W| \cdot \frac{k-1}{2} -n_0 \cdot \frac{k-1}{2}$ red edges.
        We get that the number of red edges is at least
        \begin{align}\label{eq:K_k-factor Case 2}
            &|J|\cdot \lambda_k m + |I_{red}|\cdot \frac{2m}{k} \cdot \binom{k}{2}+|I_{blue}|\cdot \frac{2m}{k} \cdot\left\lfloor \frac{k}{2} \right\rfloor 
            \left\lceil \frac{k}{2} \right\rceil+
            (1-\eps)|W| \cdot \frac{k-1}{2} -n_0 \cdot \frac{k-1}{2}.
        \end{align}
        We now plug into \eqref{eq:K_k-factor Case 2} the following facts: $m|J| = n-M$; $2m|I_{blue}| = \alpha M$ (by definition of $\alpha$); $2m|I_{red}| = M - |W| - 2m|I_{blue}| = (1-\alpha)M - |W|$, and therefore $|I_{red}| \cdot \frac{2m}{k} \cdot \binom{k}{2} = ((1 - \alpha)M - |W|) \cdot \frac{k-1}{2}$; and finally, $n_0 \cdot \frac{k-1}{2} = O(1)$. Plugging all of these, we see that \eqref{eq:K_k-factor Case 2} is at least
        \begin{align}\label{eq:K_k-factor Case 2 (ii)}
            &(n-M)\lambda_k+
            (1-\alpha)M \cdot \frac{k-1}{2} + \alpha M \cdot \frac{1}{k}\cdot\left\lfloor \frac{k}{2} \right\rfloor 
            \left\lceil \frac{k}{2} \right\rceil- \eps |W| \cdot \frac{k-1}{2} - O(1).
        \end{align}
        Note that $\eps |W| \cdot \frac{k-1}{2} + O(1) \leq \varepsilon kn$. Also, as 
        $\frac{k-1}{2} \geq \frac{1}{k} \cdot \left\lfloor \frac{k}{2} \right\rfloor 
        \left\lceil \frac{k}{2} \right\rceil$ for $k \geq 2$, the function 
        $$f(\alpha) := (1-\alpha) \cdot \frac{k-1}{2} + \alpha \cdot \frac{1}{k} \cdot \left\lfloor \frac{k}{2} \right\rfloor 
        \left\lceil \frac{k}{2} \right\rceil$$ is monotone decreasing in $\alpha$. As $\alpha \leq 2/3$, we have 
        $f(\alpha) \geq f(2/3) \geq \frac{k-1}{6} + \frac{2}{3} \cdot \frac{1}{2} \cdot \left\lfloor \frac{k}{2} \right\rfloor \geq \frac{k-1}{6} + \frac{k-1}{6} = \frac{k-1}{3} \geq \lambda_k$, using Lemma \ref{lem:lambda_k}.
        Plugging this into \eqref{eq:K_k-factor Case 2 (ii)}, we see that \eqref{eq:K_k-factor Case 2 (ii)} is at least
        $$
        (n-M)\lambda_k + \lambda_k \cdot M - \varepsilon kn = (\lambda_k - \varepsilon k)n.
        $$
        As $\varepsilon > 0$ was arbitrary, this completes the proof. 
    \end{proof}

    \subsection{$2$-factors: Proof of Theorem \ref{prop:2-factor}}
    We want to proceed similarly as in the proof for $K_k$-factors. However, a $2$-factor may not nicely embed into the constant sized parts $V_1,\ldots,V_s$ as some of its cycles may be very long. Thus, we first need to cut the cycles up into smaller parts while losing only a small number of edges, so that we can embed each of these now constant-sized parts separately. By choosing $k$ large enough in the following lemma, we ensure that the loss of edges is not substantial.
    \begin{lemma}\label{lem: cycle partition}
        Let $k$ be an integer and $F$ be a $2$-factor on $n$ vertices. Then there exists a partition of $V(F)$ into parts $U_0,U_1,\ldots,U_{\lfloor n/k\rfloor}$ such that $|U_i|=k$ for all $1 \leq i \leq \lfloor n/k \rfloor$ (hence $|U_0| < k$), and such that for every $0 \leq i \leq \lfloor n/k \rfloor$, the graph induced by $U_i$ is a disjoint union of cycles and at most $2$ paths. In particular, for every $1 \leq i \leq \lfloor n/k \rfloor$, $U_i$ contains at least $k-2$ edges. 
    \end{lemma}
    \begin{proof}
        Write $V(F) = \{v_1,\ldots,v_n\}$ such that the vertices of each cycle in $F$ appear consecutively and in the corresponding order. Let $U_i = \{v_{(i-1)k+1},\ldots,v_{ik}\}$ for $i \geq 1$, and 
        $U_0 = \{v_{(k-1)\lfloor n/k\rfloor+1},\ldots,v_n\}$.
    \end{proof}
    As in the proof of Theorem~\ref{thm: Kk-factors}, we need to investigate the different ways of embedding $2$-factors into the elements of $\mathcal{F}_m$ or their subgraphs. We again consider an imbalanced bipartite construction.
    \begin{lemma}\label{lem: bip construction 2-factor}
        Let $F$ be a $2$-factor on $3k$ vertices. A bipartite construction $B$ with ratio $\frac{1}{3}$ on $3k$ vertices contains a copy of $F$ with at least $2k-1$ red edges and a copy of $F$ with at least $2k-1$ blue edges.
    \end{lemma}
    \begin{proof}
        Without loss of generality, suppose that the red edges form a complete graph on $2k$ vertices in $B$. Write $V(F) = \{v_1,\ldots,v_{3k}\}$ such that the vertices of each cycle in $F$ appear consecutively and in the corresponding order. Consider an embedding of $F$ into $B$ such that $v_1,\ldots,v_{2k}$ get mapped into the red complete graph. It is easily seen that this gives an embedding of $F$ with at least $2k-1$ red edges. On the other hand, $F$ has an independent set $S$ of size $k$ since it is $3$-colorable. Now, consider an embedding of $F$ into $B$ such that $V(F)\setminus S$ gets mapped into the red complete graph. As $F$ is $2$-regular, this gives an embedding of $F$ with $2k$ blue edges.
    \end{proof}
    \begin{lemma}\label{lem: 2 cliques 2-factor}
        Let $F$ be a $2$-factor on $4k$ vertices and let $C$ be a red/blue-coloring of $K_{4k}$ such that the red edges form two disjoint cliques of size $2k$ each. Then $C$ contains a copy of $F$ with at least $4k-2$ red edges and a copy of $F$ with at least $8k/3$ blue edges.
    \end{lemma}
    \begin{proof}
        Let $X,Y$ denote the vertex set of the two red cliques.
        Write $V(F) = \{v_1,\ldots,v_{4k}\}$ such that the vertices of each cycle in $F$ appear consecutively and in the corresponding order. Consider an embedding of $F$ into $B$ such that $v_1,\ldots,v_{2k}$ get mapped into $X$. Note that all the cycles of $F$ which do not contain $v_{2k}$ are embedded into either $X$ or $Y$ and, hence, entirely red. The cycle of $F$ containing $v_{2k}$ possibly contains two blue edges as it may be split between $X$ and $Y$. Hence, this embedding of $F$ has at least $4k-2$ red edges.         

        Recall that the blue edges form a balanced complete bipartite graph in $C$. For the second embedding, note that $4k$ is even, and hence the number of odd cycles in $F$ is even as well. Then, we can make $F$ balanced bipartite by removing one edge from every odd cycle, yielding an embedding into the subgraph of $C$ induced by the blue edges. Since each cycle has length at least $3$, it follows that we removed at most $4k/3$ edges. Hence, the embedding has at least $e(F) - 4k/3 = 8k/3$ blue \nolinebreak edges.
    \end{proof}
    \begin{proof}[Proof of Theorem~\ref{prop:2-factor}]
        Fix an arbitrary $\eps>0$ and set $k=\lceil 6/\eps\rceil$. We show that for $n$ large enough, every red/blue-coloring of $K_n$ contains a copy of $F$ with at least $(2/3-\eps)n$ edges in some color. Let $B,\overline{B}$ be the two bipartite constructions on $3k$ vertices as in Lemma~\ref{lem: bip construction 2-factor}. 
        Let $D,\overline{D}$ be the bipartite constructions on $4k$ vertices with ratio $\frac{1}{2}$ (so $D$ contains $B$ and $\overline{D}$ contains $\overline{B}$).
        Let $C$ be as in Lemma~\ref{lem: 2 cliques 2-factor} and let $\overline{C}$ be obtained from $C$ by switching the colors (so $C,\overline{C}$ have $4k$ vertices). Then $\mathcal{F}_{2k} = \{D,\overline{D},C,\overline{C}\}$.
        
        Let $U_0,U_1,\ldots,U_{\lfloor n/k\rfloor}$ be a partition of $V(F)$ as in Lemma~\ref{lem: cycle partition}. 
        Then there exists a $2$-factor $F'$ on $V(F)$, obtained from $F$ by deleting and adding at most $4 \left\lceil \frac{n}{k} \right\rceil \leq \frac{5n}{k}$ edges, such that each $U_i$ ($0 \leq i \leq \lfloor \frac{n}{k} \rfloor$) induces a $2$-factor in $F'$. 
        
        Considering the given red/blue coloring of $K_n$, we get by Theorem~\ref{thm: get small graph} that either some color has at most $\frac{1}{3}\cdot \binom{n}{2}$ edges, or the coloring contains an element of $\mathcal{F}_{2k}$, and hence a copy $H$ of $B$, $\overline{B}$, $C$ or $\overline{C}$. 
        We write $V_1=V(H)$ and remove $V_1$ from the vertices of the complete graph. Repeating this for as long as possible, we eventually get disjoint sets of vertices $V_1,\ldots,V_s$ and a set $W$ of remaining vertices such that each $V_i$ induces either $B$, $\overline{B}$, $C$ or $\overline{C}$. Additionally, we have that either $|W|\leq n_0$ or some color $c$ has at most $\frac{1}{3} \binom{|W|}{2}$ edges on $W$. If we are in the latter case, without loss of generality, suppose that $c$ is blue. Let $I_{red}\subseteq [s]$ be the set of indices $i$ such that $V_i$ induces a copy of $C$, define $I_{blue}$ similarly for $\overline{C}$, and let $J=[s]\setminus (I_{red}\cup I_{blue})$. 
        
        We construct a copy of $F'$ with many red edges. 
        Note that $n-|W| = 3k \cdot |J| + 4k \cdot (|I_{red}| + |I_{blue}|)$. Fix a subset $I \subseteq \{1,\dots,\lfloor \frac{n}{k}\rfloor\}$ of size $\frac{n-|W|}{k}$, and fix $\tau : I \rightarrow [s]$ such that for every $j \in [s]$, $|\tau^{-1}(j)| = |V_j|/k$ (namely, $|\tau^{-1}(j)| = 3$ if $j \in J$ and $|\tau^{-1}(j)|=4$ if $j \in I_{red}\cup I_{blue}$). In other words, we assign to each $V_j$ a collection of 3 or 4 of the sets $U_i$ to be embedded into $V_j$. 
        
        For each $j\in [s]$, let $F_j$ denote the subgraph of $F'$ induced by $\bigcup_{i\in\tau^{-1}(j)}U_i$, and let $F_W$ be the subgraph of $F'$ obtained by deleting $\bigcup_{j \in [s]}V(F_j)$. Note that $F_j$ is a $2$-factor for every $j \in [s]$, and so is $F_W$. By Lemma~\ref{lem: bip construction 2-factor}, for each $j\in J$, there exists an embedding of $F_j$ into $V_j$ with at least $2k-1$ red edges. By Lemma~\ref{lem: 2 cliques 2-factor}, for every $i\in I_{red}$ there exists an embedding of $F_j$ into $V_j$ with at least $4k-2$ red edges, and for every $i\in I_{blue}$ there exists an embedding of $F_j$ into $V_j$ with at least $8k/3$ red edges. Finally, by Lemma~\ref{lem: random embedding}, there exists an embedding of $F_W$ into $W$ with at least 
        $2/3\cdot e(F_W) - n_0 = 2/3\cdot|W| - n_0$ red edges. Combining all these embeddings to an embedding of $F'$, we get a copy of $F'$ in which the number of red edges is at least
        \begin{align*}
            &|J|\cdot (2k-1)+|I_{red}|\cdot (4k-2)+|I_{blue}|\cdot \frac{8k}{3}+\frac{2}{3}\cdot|W| - O(1)
            \geq \\ & \frac{2}{3}\cdot\bigg(|J|\cdot 3k + |I_{red}|\cdot 4k+|I_{blue}|\cdot 4k+|W|\bigg)-|J| - O(1) = \frac{2n}{3} - |J| - O(1).
        \end{align*}
        Note that $|J| \leq \frac{n}{3k}$. Hence, the above is at least $(\frac{2}{3} - \frac{1}{k})n$. 
        Now, since $F$ can be obtained from $F'$ by deleting and adding at most $\frac{5n}{k}$ edges, this gives an embedding of $F$ into $K_n$ with at least $(\frac{2}{3} - \frac{6}{k})n \geq (\frac{2}{3} - \varepsilon)n$ red edges, as required. 
    \end{proof}

\section{Concluding remarks}
In this paper we proved lower bounds on the discrepancy of graphs with bounded maximum degree (Theorem~\ref{thm:max degree}), regular graphs (Theorem~\ref{thm:d-regular}), $K_k$-factors (Theorem~\ref{thm: Kk-factors}) and $2$-factors (Theorem~\ref{prop:2-factor}). Each of these bounds is in some sense tight.

The natural general question underlying these results is to find (or estimate), for each $n$-vertex graph $F$, the maximum $t = t(F)$ such that every 2-coloring of $K_n$ has a copy of $F$ with discrepancy at least $t$. 
Short of a complete solution, it would also be interesting to understand how different parameters of $F$ affect $t(F)$. 
For example, Theorem \ref{thm:max degree} addresses this question for the maximum degree. We wonder how the minimum degree of $F$ affects $t(F)$.
\begin{problem}\label{prob:min degree}
    For $\varepsilon > 0$ and integers $\delta,n \geq 1$, what is the maximum $t = t(\varepsilon,\delta,n)$ such that for every $n$-vertex graph $F$ with maximum degree at most $(1-\varepsilon)n$ and minimum degree at least $\delta$, every 2-coloring of $K_n$ has a copy of $F$ with discrepancy at least $t$?
\end{problem}
\noindent
It seems plausible that $t(\varepsilon,\delta,n)$ grows substantially with $\delta$ (recall that this is not the case if we replace $\delta$ with the average degree $d$ (at least for a large range of $d$); see the paragraph after Theorem \nolinebreak \ref{thm:max degree}). 


It might also be interesting to characterize the 2-colorings of $K_n$ which minimize the discrepancy of $F$-copies, for various choices of $F$. For example, to show that Theorems \ref{thm:max degree} and \ref{thm:d-regular} are tight, we used the random coloring and the complete bipartite graph, respectively (see the paragraphs following these theorems). Are these colorings the only extremal ones for these problems, i.e., are there stability versions of these theorems?

Several prior works \cite{BCL,FHLT:21,GKM_Hamilton,GKM_trees,HLMP} studied discrepancy in the multicolor setting. In this setting, one considers $q$-colorings $q : E(H) \rightarrow [q]$ of a graph $H$ (we will take $H = K_n$), and the discrepancy of a subgraph $F \subseteq E(H)$ is the maximum $t$ such that there exists a color $i \in [q]$ for which the number of edges of $F$ in color $i$ is at least $\frac{e(F)+t}{q}$. (For $q=2$, this coincides with the definition of discrepancy from Section \ref{sec:introduction}.) 
Our proofs of Theorems \ref{thm:max degree} and \ref{thm:d-regular} can be adapted to work for any number $q$ of colors. First, one proves an analogue of Lemma \ref{lem:cut discrepancy} stating that for a coloring $q : E(K_n) \rightarrow [q]$, if there is no copy of $F$ with high discrepancy then for every partition $V(K_n) = X \cup Y$ with $|X| = \lfloor \frac{n}{2} \rfloor$, all colors appear roughly the same number of times in the bipartite graph $(X,Y)$. Then, in all other proofs (e.g., the proof of Lemma \ref{lem:main} or the proof of Theorem \ref{thm:max degree}), one transforms the $q$-coloring into a 2-coloring by identifying colors $2,\dots,q$ into one color, say $\star$. Let $G$ be the graph of edges of color 1. The key point is that in all of these proofs, one finds a copy of $F$ with high discrepancy by finding two copies $F_1,F_2$ of $F$ such that 
$\big| |E(F_1) \cap E(G)| - |E(F_2) \cap E(G)| \big| \geq s$ for some parameter $s$. 
Suppose without loss of generality that 
$|E(F_1) \cap E(G)| \geq |E(F_2) \cap E(G)|$.
Observe that either 
$|E(F_1) \cap E(G)| \geq \frac{e(F)+s}{q}$ (meaning that $F_1$ has discrepancy at least $s$), or 
$|E(F_2) \setminus E(G)| \geq (q-1)\frac{e(F)+s}{q}$. In the latter case, by passing back to the original $q$-coloring, we see that one of the colors $2,\dots,q$ is present on at least $\frac{e(F)+s}{q}$ of the edges of $F_2$, so $F_2$ has high discrepancy, as required. We omit further details. 


\vspace{1cm}

\noindent
{\bf Acknowledgments:} We thank Zach Hunter for useful discussions regarding the proof of Lemma \nolinebreak \ref{lem:sqrt deviation}.

\bibliographystyle{abbrv}
\bibliography{library}

\appendix

\section{Proof of Lemma \ref{lem:hypergeometric}}

We will deduce the lemma from the following statement: For every integer $t$ with 
$$
|pk - t| \leq 0.5\sqrt{p(1-p) \cdot \min(k,n-k)},
$$
it holds that 
\begin{equation}\label{eq:anticoncentration main}
\mathbb{P}[|A \cap P| = t] \geq 
0.14 \cdot \sqrt{\frac{\eta}{p(1-p)k}}.
\end{equation}
To deduce the lemma from \eqref{eq:anticoncentration main}, note that $n - k \geq \eta k$, and so \eqref{eq:anticoncentration main} holds whenever $t = pk + s$ for $|s| \leq 0.5\sqrt{p (1-p) \eta k}$. Using this for all 
$0.2 \sqrt{p (1-p) \eta k} \leq s \leq 0.5\sqrt{p (1-p) \eta k}$, 
we get that 
$$
\mathbb{P}\big[ |A \cap P| \geq pk + 0.2 \sqrt{p (1-p) \eta k} \big]
\geq 0.3 \sqrt{p (1-p) \eta k} \cdot 0.14 \cdot \sqrt{\frac{\eta}{p(1-p)k}} \geq 
0.04 \eta.
$$
As $p(1-p) \geq \eta(1-\eta) \geq 0.5\eta$, the above is also a lower bound on the probability of having 
$|A \cap P| \geq pk + 0.1\eta\sqrt{k}$.
The bound for 
$\mathbb{P}\big[ |A \cap P| \leq pk - 0.1\eta\sqrt{k} \big]$
is obtained the same way. 

Now we prove \eqref{eq:anticoncentration main}. 
\begin{equation}\label{eq:anticoncentration 0}
    \mathbb{P}[|A \cap P| = t] = \frac{\binom{pn}{t}\binom{(1-p)n}{k-t}}{\binom{n}{k}} 
\end{equation}
We use Stirling's approximation to estimate the above binomial coefficients (using that $n \geq k \gg 1/\eta$, $n-k \geq \eta n$ and $t = pk + O(\sqrt{k}$).
\begin{align}\label{eq:anticoncentration 1}
\binom{pn}{t} &= \nonumber \frac{(pn)!}{t! (pn-t)!} = 
\frac{1 + o_k(1)}{\sqrt{2\pi}} \cdot 
\sqrt{\frac{pn}{t(pn-t)}} \cdot 
\frac{(\frac{pn}{e})^{pn}}{(\frac{t}{e})^t \cdot (\frac{pn-t}{e})^{pn-t}} 
\\ &\geq
\frac{1 + o_k(1)}{\sqrt{2\pi}} \cdot \sqrt{\frac{pn}{t(pn-t)}} \cdot 
\frac{(pn)^{pn}}{t^t (pn-t)^{pn-t}}.
\end{align}
Write $t = pk + s$, so that $|s| \leq \sqrt{0.25\eta(1-\eta)\cdot \min(k,n-k)}$ by assumption. We have
$$
\frac{pn}{t(pn-t)} = \frac{pn}{(pk+s)(p(n-k)-s)} = 
\frac{1}{pk + o(k)}
$$
and
\begin{align}\label{eq:anticoncentration 1.1}
\frac{(pn)^{pn}}{t^t (pn-t)^{pn-t}} \nonumber &= 
\frac{(pn)^{pn}}{(pk+s)^{pk+s} (p(n-k)-s)^{p(n-k)-s}} 
\\ &= 
\frac{(pn)^{pn}}{(pk)^{pk}(p(n-k))^{p(n-k)}} \cdot 
\left( \frac{p(n-k)-s}{pk+s}\right)^{s} \cdot \left( 1 + \frac{s}{pk} \right)^{-pk} \cdot \left( 1 - \frac{s}{p(n-k)} \right)^{-p(n-k)}
\end{align}
In the first term above, we may divide through by $p^{pn}$. Also, using the inequality $1+x \leq e^x$ (which holds for all $x \in \mathbb{R}$), we can bound the third and fourth terms as follows:
$$
\left(1 + \frac{s}{pk} \right)^{pk} \leq e^{s} \; \; \; \text{and} \; \; \; 
\left( 1 - \frac{s}{p(n-k)} \right)^{p(n-k)} \leq e^{-s},
$$
and hence,
$$
\left( 1 + \frac{s}{pk} \right)^{-pk} \cdot \left( 1 - \frac{s}{p(n-k)} \right)^{-p(n-k)} \geq 
 e^{-s} \cdot e^s = 1.
$$
Therefore, \eqref{eq:anticoncentration 1.1} is at least
$$
\frac{n^{pn}}{k^{pk}(n-k)^{p(n-k)}} \cdot 
\left( \frac{p(n-k)-s}{pk+s}\right)^{s}.
$$
Plugging this into \eqref{eq:anticoncentration 1}, we get
\begin{equation}\label{eq:anticoncentration 1 final}
\binom{pn}{t} \geq \frac{1 + o_k(1)}{\sqrt{2\pi}} \cdot \sqrt{\frac{1}{pk+o(k)}} \cdot \frac{n^{pn}}{k^{pk}(n-k)^{p(n-k)}} \cdot 
\left( \frac{p(n-k)-s}{pk+s}\right)^{s}.
\end{equation}
By symmetry with respect to replacing $p$ with $1-p$ and $t$ with $k-t = (1-p)k - s$ (so $s$ is replaced with $-s$), we similarly have
\begin{align}\label{eq:anticoncentration 2 final}
    \binom{(1-p)n}{k-t} \geq 
    \frac{1 + o_k(1)}{\sqrt{2\pi}} \cdot \sqrt{\frac{1}{(1-p)k + o(k)}} \cdot  
    \frac{n^{(1-p)n}}{k^{(1-p)k}(n-k)^{(1-p)(n-k)}}  
    \cdot 
    \left( \frac{(1-p)(n-k)+s}{(1-p)k-s} \right)^{-s}.
\end{align}
Note that
\begin{equation}\label{eq:anticoncentration 3}
    \left( \frac{p(n-k)-s}{pk+s}\right)^{s} \cdot 
    \left( \frac{(1-p)(n-k)+s}{(1-p)k-s} \right)^{-s} = 
    \left( \frac{(1-p)k - s}{pk+s} \right)^s \cdot 
    \left( \frac{p(n-k) - s}{(1-p)(n-k) + s} \right)^s.
\end{equation}
We now use the inequality $1+x \geq e^{x - x^2}$, which holds for all $x$ with $|x| \leq 0.5$. We have
\begin{align*}
    \left( \frac{(1-p)k - s}{pk+s} \right)^s &= 
    \left( \frac{1-p}{p} \right)^s \left( 1 - \frac{s}{(1-p)(pk+s)} \right)^s
    \\ &\geq
    \left( \frac{1-p}{p} \right)^s \cdot \exp\left( -\frac{s^2}{(1-p)(pk+s)} - \frac{s^3}{(1-p)^2(pk+s)^2} \right)
    \\ &\geq 
    \left( \frac{1-p}{p} \right)^s \cdot e^{-1/2}.
\end{align*}
Here, the last inequality uses that 
$|s| \leq 0.5\sqrt{p(1-p)k}$, $\eta \leq p \leq 1-\eta$, and $k$ is large enough. 
Similarly,
\begin{align*}
    \left( \frac{p(n-k) - s}{(1-p)(n-k) + s} \right)^s &= 
    \left( \frac{p}{1-p} \right)^s \cdot \left( 1 - \frac{s}{p((1-p)(n-k) - s)} \right)^s
    \\ &\geq
    \left( \frac{p}{1-p} \right)^s \cdot 
    \exp\left( - \frac{s^2}{p((1-p)(n-k)-s)} - \frac{s^3}{p^2((1-p)(n-k)-s)^2}\right)
    \\ &\geq 
    \left( \frac{p}{1-p} \right)^s \cdot e^{-1/2}.
\end{align*}
Here we used that $|s| \leq 0.5\sqrt{p(1-p)(n-k)}$.
Combining the above with 
\eqref{eq:anticoncentration 1 final}, \eqref{eq:anticoncentration 2 final} and \eqref{eq:anticoncentration 3}, we get
$$
\binom{pn}{t} \cdot \binom{(1-p)n}{k-t} \geq \frac{1 + o_k(1)}{2\pi} \cdot \sqrt{\frac{1}{p(1-p)k^2 + o(k^2)}} \cdot \frac{n^n}{k^k(n-k)^{n-k}} \cdot e^{-1}.
$$
Finally, we estimate $\binom{n}{k}$, as follows.
$$
\binom{n}{k} = \frac{n!}{k!(n-k)!} \leq 
\frac{1+o_k(1)}{\sqrt{2\pi}} \cdot \sqrt{\frac{n}{k(n-k)}} \cdot \frac{(\frac{n}{e})^n}{(\frac{k}{e})^k(\frac{n-k}{e})^{n-k}}.
$$
Plugging the above into \eqref{eq:anticoncentration 0}, we get
$$
\mathbb{P}[|A \cap P| = t] \geq \frac{e^{-1} + o_k(1)}{\sqrt{2\pi}} \cdot \sqrt{\frac{k(n-k)}{\left( p(1-p)k^2 + o(k^2) \right) n}} \geq 
0.14 \cdot \sqrt{\frac{\eta}{p(1-p)k}}.
$$
Here we used that $k \leq (1-\eta)n$.
\end{document}